\documentclass[dvipsnames]{article}

\usepackage[latin1, utf8]{inputenc}
\usepackage[T1]{fontenc}
\usepackage[english]{babel}
\usepackage{enumerate}
\usepackage{todonotes}
\usepackage{hyperref}
\usepackage{amsmath}
\usepackage{amssymb}
\usepackage{amsthm}
\usepackage{mathrsfs}
\usepackage{authblk}

\usepackage{xkeyval}
\usepackage{stmaryrd}
\usepackage{tocbibind}
\usepackage{comment}
\usepackage{caption}
\usepackage{subcaption}
\usepackage{tikz-cd} 
\usepackage[left=2.5cm,top=3cm,right=2.5cm,bottom=3cm,bindingoffset=0.5cm]{geometry}

\newcommand{\R}{\mathbb{R}}
\newcommand{\N}{\mathbb{N}}

\newcommand{\sing}{\mathrm{sing}}
\newcommand{\adhco}{\overline{\mathrm{conv}}}

\newcommand{\cone}{\operatorname{cone}}
\newcommand{\ran}{\operatorname{ran}}

\newcommand{\scl}[2]{\left\langle #1, #2\right\rangle}

\newcommand{\ext}{\mathrm{ext}}
\newcommand{\dom}{\mathrm{dom}}

\newcommand{\ie}{\textit{i.e., }}

\DeclareMathOperator*{\argmax}{arg\,max}
\DeclareMathOperator*{\argmin}{arg\,min}

\newcommand{\cp}[1]{{\color{black}#1}}
\newcommand{\cz}[1]{{\color{black}#1}}

\newcommand {\e}  {\varepsilon}
\newcommand{\vertiii}[1]{{\vert\kern-0.25ex\vert\kern-0.25ex\vert #1 
    \vert\kern-0.25ex\vert\kern-0.25ex\vert}}

\newtheorem{prpstn}{Proposition}[section]
\newtheorem{lmm}{Lemma}[section]
\newtheorem{thrm}{Theorem}[section]
\newtheorem{dfntn}{Definition}[section]
\newtheorem{crllr}{Corollary}[section]
\newtheorem{rmrk}{Remark}[section]
\newtheorem{assum}{Assumption}[section]

\title{
\cz{Generalisation of Farkas' lemma beyond closedness: a constructive approach via Fenchel-Rockafellar duality}
}

\author[a]{Camille Pouchol}
\author[b,c]{Emmanuel Trélat}
\author[b,c]{Christophe Zhang}
\affil[a]{Université Paris Cité, CNRS, MAP5, F-75006 Paris, France.}
\affil[b]{Sorbonne Universit\'{e}, Universit\'{e} Paris Cit\'{e}, CNRS, Laboratoire Jacques-Louis Lions, 75005 Paris, France.} 
\affil[c]{CAGE, INRIA, Paris, France.}
\date{\empty}
\begin{document}
\maketitle

\begin{abstract}

Farkas' lemma is an ubiquitous tool in optimisation, as it provides necessary and sufficient conditions to have $b \in A(P)$, where $P$ is a closed convex cone, $A$ is a (continuous) linear mapping and $b$ is a fixed vector. 
The standard underlying hypothesis is the closedness of $A(P)$, which is not always satisfied and can be difficult to check. We devise a new method to generalise Farkas' lemma, based on a primal-dual pair of optimisation problems and Fenchel-Rockafellar duality theory. We work under the sole hypothesis that $P$ be generated by a closed bounded convex set. 
This hypothesis is weaker than in previous generalisations of Farkas' lemma, which almost all require that $A(P)$ be closed, or, in few cases, that only $P$ be closed.
In our case, $P$ (and a fortiori $A(P)$) is not necessarily closed; we uncover necessary and sufficient conditions both for $b \in A(P)$ and $b \in \overline{A(P)}$. For a given $\e \geq 0$, we exhibit constructive characterisations of $x \in P$ such that $\|Ax-b\| \leq \e$ when it exists, by means of optimality conditions. For $\e = 0$, these strongly rely on whether the dual problem admits a solution, and we discuss conditions under which it does.
Finally, we also explain how, upon relaxation, we may apply our method to a nonconvex cone.

\end{abstract}
\tableofcontents
\section{Introduction and main results}
Let $X$ and $Y$ be two (real) Hilbert spaces, $A : X \to Y$ a continuous bounded operator (an element of $\mathcal L(X,Y)$), $b \in Y$ and $P \subset X$ a cone containing $0$. We are interested in the following problem, which is a variant of Farkas' celebrated lemma (\cite{Farkas1902}): do we have $b\in A(P)$, that is, is the vector $b$ in the (somewhat) implicit cone $A(P)$?

This problem can also be formulated as a constrained linear equation:
\begin{equation}\label{farkas-problem} \text{Solving the linear equation $A x = b$ under the constraint $x \in P$}. \end{equation} We let $A (P) = \{A x, \, x\in P\}$.

We consider both 
\begin{itemize}
\item[(i)] the \textbf{approximate solvability} problem, 
\begin{equation}
    \label{approx-Farkas-problem}
    \textrm{For any given} \ \e>0, \ \textrm{find} \ x_\e \in P \ \textrm{such that} \ \|A x_\e-b\|_Y\leq \e,
\end{equation}
that is, determining whether $b\in \overline{A(P)}.$
\item[(ii)] the \textbf{exact solvability} problem, 
\begin{equation}
    \label{exact-Farkas-problem}
    \textrm{Find} \ x \in P\ \textrm{such that} \ A x = b,
\end{equation}
that is, determining whether $b\in A(P)$.
\end{itemize}

\subsection{Farkas' lemma(s): necessary and sufficient conditions for solvability}
The original version of Farkas' lemma~\cite{Farkas1902} is concerned with the  exact solvability problem within the cone $\R_+^n$ in $\R^n$. It provides a necessary and sufficient condition (often formulated as an alternative):
\begin{lmm}
     Let $f, v_1, \cdots, v_m \in \R^n$. Then, there exist $(f_1, \ldots, f_m) \in \R_+^m$ such that
    \[f=\sum_{i=1}^m f_i v_i\]
    if and only if
    \[\{d\in \R^n, \, \langle v_i, d\rangle \leq 0, \ \forall i \in \{1, \ldots, m\}\} \subset \{d\in \R^n, \, \langle d, f\rangle \leq 0\}. \]
\end{lmm}

For a survey on the rich history of Farkas' lemma and its generalisations, we refer to~\cite{Dinh2014Farkas}.
The most standard and used ones writes as the following alternative, with $P^\circ = \{x \in \R^n, \; \forall y \in P, \; \langle x, y \rangle \leq 0\}$ standing for the polar cone of $P$. 
\begin{lmm}
\label{lmm-farkas}
    [Farkas' lemma]
    Let $A\in \R^{m\times n}$, $b\in \R^m$ and $P\subset \R^n$ be a closed convex cone such that $A(P)$ is closed. Then, given $b \in \R^m$, the following alternative holds: 
    \begin{enumerate}
        \item the exact solvability problem admits a solution,
        \item there exists $y\in \R^m$ such that
        $A^\ast y \in P^\circ$ and $\langle b, y \rangle >0.$
    \end{enumerate}
\end{lmm}

Farkas' lemma has been generalised in several directions: in particular, it can be stated in very abstract general infinite-dimensional settings~\cite{craven1977generalizations}.
However, the hypothesis that $A(P)$ is closed remains a cornerstone to all these results. In fact, in the above cited work, the authors prove that the Farkas alternative holds if and only if $A(P)$ is closed. Under this key assumption, let us note that distinguishing between exact and approximate solvability as we did above is irrelevant. 

Up to our knowledge, only the two works~\cite{HLermaLasserre1997, lasserre1997farkas} formulate a generalisation without assuming the closedness of $A(P)$. Let us give one such result, in the particular case of reflexive Banach spaces; here we translate it into the Hilbertian setting we are interested in.
\begin{thrm}\label{thm-Lasserre}
        [Hernandez-Lerma, Lasserre, 1997]
        Assume that $P\subset X$ is a closed convex cone. Then, the following are equivalent for $b\in Y$:
        \begin{itemize}
            \item[(a)] The equation $Ax=b$ has a solution $x\in P$;
            \item[(b)] There exists $M>0$ such that, for all $y\in Y, \ p\in P^\circ$,
            \[\langle b, y \rangle_Y + M\|A^\ast y +p \|_X \geq 0.\]
        \end{itemize}
    \end{thrm}
Given the equivalence proved in \cite{craven1977generalizations}, foregoing the restrictive hypothesis that $A(P)$ be closed requires a stronger characterisation. Accordingly, the above theorem gives a quantified necessary sufficient condition. The underlying idea is to apply the "usual" Farkas lemma to an augmented system. The corresponding proof is rather intricate, it is not constructive, and the resulting condition, though quantified, is difficult to put in use since the inequality hinges on two variables.


\subsection{Main contributions}

This article presents a generalisation of Farkas' lemma, close in spirit to Theorem~\ref{thm-Lasserre}, which we state under hypotheses of increasing generality on the cones involved. As in~\cite{HLermaLasserre1997, lasserre1997farkas}, we forego the hypothesis that $A(P)$ is closed, hence our formulation of the two distinct problems \eqref{approx-Farkas-problem} and~\eqref{exact-Farkas-problem}.

Our main contributions are then as follows:
\begin{itemize}
    \item we recover the results of~\cite{HLermaLasserre1997, lasserre1997farkas} by means of an arguably simpler proof,
    \item not only do we forego of the assumption that $A(P)$ is closed (which is the usual hypothesis in Farkas' lemma, see Lemma \ref{lmm-farkas} and \cite{craven1977generalizations}), we also weaken the underlying hypothesis that $P$ itself is closed: \cz{our approach only requires that $P$ can be generated by a bounded closed convex set containing $0$},
    \item contrarily to other proofs, our proof is constructive under general hypotheses, in a sense we shall make more precise later on,
    \item the overall technique lends itself to relaxation, i.e., one can consider the approximate or exact solvability problem in nonconvex cones, and apply our strategy on well-chosen convex cones derived from them.
\end{itemize}

\cz{\paragraph{Algorithmic viewpoint.}
A key feature of our approach is that (approximate or exact) Farkas solutions can be \emph{constructed} by solving an explicit convex optimisation problem. For a tolerance $\varepsilon\geq 0$, we minimise an \textbf{unconstrained} convex functional over $Y$ (the dual problem $\mathcal{D}_{\e}$). Then, under suitable conditions, we may recover a primal variable $x_\varepsilon$ through explicit first-order optimality conditions (projection onto a cone in the closed-cone case, or a support-function subgradient formula in the general case). This yields a practical primal--dual procedure to build either a solution $x\in P$ with $Ax=b$, or approximate solutions $x_\e \in P$ with $\Vert Ax_\e-b\Vert_Y\leq \varepsilon$, when they exist.}

\paragraph{Related literature.}
Farkas-type statements and theorems of the alternative can be derived from convex duality in various settings (see, e.g., Farkas-type results for constrained optimisation problems, obtained via conjugate functions and conic duality in \cite{bot2005farkas, boct2006farkas, boct2007some}).

In this article however, we address the classical problem of solving a linear equation on a (convex) cone, by formulating a well-chosen auxiliary optimisation problem. 
\cz{Problems of this form are already present in the works~\cite{zualinescu2008zero, zualinescu2012duality}:  given $c \in X$ fixed, $x \mapsto \langle c, x \rangle_X$ is maximised under the constraints $Ax =b$ and $x \in P$. A duality framework is then established to prove Farkas-type results on the solvability of the constrained equation $Ax=b$, $x\in P$.

In contrast to the above, we provide a unified \emph{constructive} approach based on Fenchel--Rockafellar duality in a Hilbertian framework. We propose a different auxiliary optimisation problem (given by~\eqref{primal-general}), for which
 our approach yields explicit primal minimisers (which are automatically solutions of the solvability problems \eqref{approx-Farkas-problem}--\eqref{exact-Farkas-problem}) from dual minimisers. Our results cover cones of the form $P_r=\mathrm{cone}(K_r)$ where the bounded convex generator $K_r$ is closed but $P_r$ (and thus $A(P_r)$) may fail to be closed.}

\medskip

Let us now introduce our results in full detail. For ease of readability, use and comparison with other results, we proceed in increasing order of generality.


\subsection{Variants of Farkas' lemma: necessary and sufficient conditions for solvability}
Let us first state necessary and sufficient conditions for the existence of solutions. Throughout the paper, $X$ and $Y$ are real Hilbert spaces, with inner products $\langle \cdot, \cdot\rangle_X$ and $\langle \cdot, \cdot\rangle_Y$ (with corresponding norms $\|\cdot\|_X$ and $\|\cdot\|_Y$), respectively. \cz{We use the Riesz representation theorem so that the adjoint $A^*$ of a bounded linear operator $A\in\mathcal L(X,Y)$ is the Hilbert adjoint $A^\ast \in\mathcal L(Y,X)$ defined by $\langle Ax, y\rangle_Y = \langle x, A^*y\rangle_X$ for all $(x,y) \in X \times Y$.}

\subsubsection{Closed convex cone}

Here, we assume that \textbf{$P$ is closed and convex}, and for $x \in X$, we let $x_+$ be its projection onto $P$.
Note that in this first setting, we already forego the usual hypothesis that $A(P)$ is closed, but retain the classical assumption of closedness for $P$, as in \cite{HLermaLasserre1997, lasserre1997farkas}.
\begin{thrm}
\label{intersect-ball-ncs}
Let $A \in \mathcal L(X,Y)$, $b \in Y$ and $P \subset X$ a closed convex cone. The following dual characterisations hold for any $b\in Y$:
\begin{itemize}
\item[(i)] \[b \in \overline{A(P)} \quad \iff \quad \forall y \in Y, \; (A^\ast y)_+ = 0 \, \implies \,\scl{b}{y}_Y \leq 0. \]
\item[(ii)] \[b \in A(P) \quad \iff \quad \exists C>0,\; \forall y \in Y,\; \scl{b}{y}_Y \leq C\,\|(A^\ast y)_+\|_X. \]
\end{itemize}
\end{thrm}
\begin{rmrk}
If $A(P)$ is closed, (i) and (ii) are equivalent and the first solvability condition associated to (i) is clearly equivalent to the negation of item 2. of the standard Farkas' Lemma~\ref{lmm-farkas}, because $z_+ = 0$ if and only if $z \in P^\circ$.

The second solvability condition, given in (ii), is a more compact and workable form of that uncovered in~\cite{HLermaLasserre1997,lasserre1997farkas}. Indeed, one can check that (ii) and Theorem~\ref{thm-Lasserre}[(b)] are actually equivalent as follows: in a Hilbert space and for a closed convex cone $P$, \cz{Moreau's decomposition $z=z_+ + z_-$ with $z_+\in P$, $z_-\in P^\circ$ and $\langle z_+,z_-\rangle=0$ yields
$$
\inf_{p\in P^\circ}\Vert z-p\Vert = \Vert z_+\Vert.
$$
Applying this identity to $z=A^*y$, the distance-to-$P^\circ$ term that appears in Theorem~\ref{thm-Lasserre}[(b)] can be rewritten as $\Vert(A^*y)_+\Vert$, leading to (ii).}
\end{rmrk}

\begin{rmrk}
\cz{
    Taking $P=X$, we recover the well-known fact that
    \[b\in \overline{\ran (A)} \iff \Big(A^\ast y=0 \implies \langle b, y \rangle_Y \leq 0 \Big) \iff \Big(A^\ast y=0 \implies \langle b, y \rangle_Y = 0 \Big) \iff b\perp \ker(A^\ast), \]
 \ie $\ker(A^\ast)^\perp= \overline{\ran (A)}$.

 On the other hand, the necessary and sufficient condition in $(ii)$ corresponds to the well-known condition of range inclusion in functional analysis (see \cite[Lemma 2.4]{li2012optimal}): denoting by $\pi_b$ the projection onto $\R b$, we have
 \[\begin{aligned}\ran(\pi_b)\subset \ran(A) \iff b\in \ran (A) &\iff \exists C>0, \forall y \in Y, \langle b, y\rangle_Y \leq C\|A^\ast y\|_X \\
 &\iff \exists C>0, \forall y \in Y, |\langle b, y\rangle_Y |\leq C\|A^\ast y\|_X \\
 &\iff \exists C>0, \forall y \in Y, \|\pi_b y\|\leq C \|A^\ast y \|_X.
 \end{aligned}\]

 Finally, we recover a well-known sufficient condition for the surjectivity of $A$:
  if there exists $C>0$ such that for $y\in Y, \|A^\ast y\|\geq c\|y\|$, then for $b\in Y$, 
  \[|\langle b, y \rangle |\leq \|b\|\|y\| \leq \frac{\|b\|}{c}\|A^\ast y\|,\]
  which implies $b\in \ran(A)$ by Theorem \ref{intersect-ball-ncs} $(ii)$.
  }
 \end{rmrk}


\subsubsection{Cone generated by a bounded closed convex set}

Note that a mandatory hypothesis for all versions of Farkas' lemma we are aware of is that $P$ is closed: this follows immediately from the usual assumption that $A(P)$ is closed \cite{craven1977generalizations}. In contrast, from this point on, we forego this hypothesis: indeed, the fundamental assumptions underlying our approach concern the choice of a generating set for the cone, rather than the cone itself. 

We now assume that the cone $P$ (denoted $P_r$ in this case) is generated by a bounded closed convex set~$K_r$ containing $0$, i.e., we have $\cone(K_r) = P_r$. 
$P_r$ is then still convex, but it need not be closed. On the other hand, if $P_r$ is closed, setting $K_r = P_r \cap \overline{B(0,1)}$ yields a bounded closed convex set containing $0$ that does generate $P_r$. Hence, this hypothesis is indeed weaker than the closedness of $P_r$, and thus also weaker than the closedness of $A(P_r)$. \cz{As hinted at by the following remark, this significantly extends the class of possible convex cones for which one can formulate a Farkas lemma.}

\begin{rmrk}
    A sufficient condition for the closedness of $P_r = \cone(K_r)$ is for $0$ to be in the interior of $K_r$ relative to $P_r$, see Lemma~\ref{closed_generated_cone}.

    \cz{As a counterexample (in $\R^2$) of a bounded closed convex set containing $0$ which generates a nonclosed cone, note that
    \[\cone\big(\overline{B((1,0), 1)}\big)=(\R_{++}\times \R) \cup \{(0,0)\},\]
    and also that
    \[\cone\left(\overline{B((1,0), 1)} \cap\overline{B((0,1), 1)}\right)=(\R_{++})^2 \cup \{(0,0)\}.\]
    Conversely, these two examples can be generalised to higher dimensions, and show that our approach can help handle exact and approximate resolution in useful non-closed cones, upon choosing appropriate generators.}
\end{rmrk}

We let $\sigma_{K_r} : x \mapsto \sup_{y \in K_r} \langle x, y\rangle_X$ be the support function of $K_r$.
\begin{thrm}
\label{general-ncs}
Let $A \in \mathcal L(X,Y)$, $b \in Y$, $K_r \subset X$ a bounded closed convex set containing $0$, and $P_r=\cone(K_r).$
The following dual characterisations hold for any $b\in Y$:
\begin{itemize}
\item[(i)] \[\begin{aligned}b \in \overline{A( P_r)} \quad &\iff \quad \forall y \in Y, \; \sigma_{K_r}(A^\ast y) = 0 \, \implies \,\scl{b}{y}_Y \leq 0, \\
&\iff \quad \forall y\in Y, \; A^\ast y \in P_r^\circ \, \implies \, \langle b, y \rangle_Y \leq 0. \end{aligned}\]
\item[(ii)] \[b \in A (P_r)  \quad \iff \quad \exists C>0,\; \forall y \in Y,\; \scl{b}{y}_Y \leq C\,\sigma_{K_r}(A^\ast y). \]
\end{itemize}
\end{thrm}
\begin{rmrk}
 The dual characterisation $(i)$ has two possible formulations due to the fact that $\sigma_{K_r}(z) = 0$ if and only if $z \in P_r^\circ$ (see Lemma~\ref{lem-separation-interpretation}). The second formulation illustrates the direct link between Theorem~\ref{general-ncs} and previous generalisations of Farkas' Lemma. 
\end{rmrk}

\begin{rmrk}  \label{rq-Kr-is-the-ball}
The two theorems~\ref{intersect-ball-ncs} and~\ref{intersect-ball-opt} above given for a closed and convex cone $P_r$ correspond to choosing $K_r = P_r \cap \overline{B(0,1)}$ and applying Theorem \ref{general-ncs}.
\end{rmrk}

\subsection{Constructive characterisations}
Our proof of the direct implications $\implies$ in theorems \ref{intersect-ball-ncs} and \ref{general-ncs} above is based on Fenchel-Rockafellar duality~\cite{Rockafellar1970}. We denote by $\delta_C$ the indicator of a set $C \subset X$, in the convex-analytic sense, and we let $j_{K_r} = \inf\{\alpha > 0, \, x \in \alpha K_r\}$ be the gauge function of $K_r$.

For $\e\geq 0$, we introduce the auxiliary functional
\[
\forall x \in X, \quad \mathcal{F}_\e(x):= \tfrac{1}{2} j_{K_r}^2(x) + \delta_{\{x \in X, \;  \|A x - b\|_Y \leq \e\}}(x).\]

Because $j_{K_r}(x) < +\infty$ if and only if $x \in P_r$,  the expression of $\mathcal{F}_\e$ is such that, for a given $\e\geq 0$
\[\forall x \in X, \quad \mathcal{F}_\e(x)< +\infty \ \iff \ \begin{cases}
    \|Ax-b\|_Y\leq \e\\
    x\in P_r
    \end{cases}.\]

A natural way of finding such an $x$ is to minimise $\mathcal{F}_\e$, we thus introduce the auxiliary optimisation problem
\begin{equation}\tag{$\mathcal{P}_{\e}$}
\label{primal-general}\inf_{x\in X} \mathcal{F}_\e(x).\end{equation}
We then have
\begin{equation}
    x^\star_\e \ \text{is a minimiser for} \ \eqref{primal-general} \ \implies \ \begin{cases}\|Ax^\star_\e - b \|_Y \leq \e, \\ x_\e^\star \in P_r. \end{cases}
        \label{why-aux-prob}
\end{equation}
Thus, solving the optimisation problem \eqref{primal-general} for all $\e>0$ solves \eqref{approx-Farkas-problem}, and solving it for $\e=0$ yields a solution to \eqref{exact-Farkas-problem}.

To solve this optimisation problem, we introduce its dual problem, following Fenchel-Rockafellar duality theory. For $\e\geq 0$, we show that it is given by
\begin{equation}
\label{dual-general}
\tag{$\mathcal{D}_{\e}$}\inf_{y \in Y} J_\e(y), \qquad J_\e(y) := 
\tfrac{1}{2} \sigma_{K_r}^2(A^\ast y) - \scl{b}{y}_Y + \e \|y\|_Y.
\end{equation}

\cz{\begin{rmrk}
    An important feature of this dual problem is that its cost function $J_\e$ is unconstrained (finite everywhere), thanks to our choice of the auxiliary optimisation problem, which involves the function $\tfrac12j_{K_r}^2$. This function has domain $P$, but its convex conjugate $\tfrac12 \sigma_{K_r}^2$ always has domain $X$.
    
    This is numerically relevant, and in contrast with the dual problems formulated in \cite{zualinescu2008zero, zualinescu2012duality}, which are linearly constrained. 
\end{rmrk}
We establish strong duality between the primal and dual problems, namely: \[\inf_{x\in X} \mathcal{F}_\e(x) = - \inf_{y \in Y} J_\e(y).\]
}

The optimality conditions derived for these primal dual pairs of optimisation problems then lead to constructive characterisations of solutions, when they exist. 

\cz{
\begin{rmrk}
    If $b \in A(P_r)$, the best constant $C= C(b)$ in Theorem~\ref{general-ncs} is given by (see Proposition \ref{ncs-exact}):
    \begin{align*} C(b) := & \inf\{j_{K_r}(x),\; Ax =b, \, x\in P_r\}. \end{align*}
    If $b\in A(P_r)$, so is any positive multiple thereof, and $C(b)$ is clearly $1$-positively homogeneous. 
    
    By strong duality, $\tfrac{1}{2}C^2(b)$ also equals $-\inf\{{\tfrac{1}{2} \sigma_{K_r}^2(A^\ast y) - \scl{b}{y}_Y , \; y \in Y\}}$. As an infimum of continuous functions, $b \mapsto C(b)$ is upper-semicontinuous on $A(P_r)$.
    
    If $P_r=X$ and $K_r = \overline{B(0,1)}$, $C(b)$ equals the minimal norm solution for the equation $Ax=b$. It also equals $\|A^\dagger b\|_X$, whenever the Moore-Penrose inverse $A^\dagger$ is well-defined (that is, when $\mathrm{Ran}(A)$ is closed).
\end{rmrk}
}

\subsubsection{Closed convex cone}

As noted in Remark \ref{rq-Kr-is-the-ball}, the case of a closed convex cone $P$ is a particular case of a cone $P_r$ generated by a bounded closed convex set $K_r$. Taking $K_r:=P\cap \overline{B(0,1)}$, the primal problem \eqref{primal-general} writes
\begin{equation}\tag{$\mathcal{P}^{\text{cl}}_\e$} \label{primal-closed}\inf_{x\in X}\mathcal{F}^{\text{cl}}_\e(x), \quad \mathcal{F}^{\text{cl}}_\e(x):=\tfrac{1}{2}\|x\|_X^2 + \delta_{P}(x) +  \delta_{\{x \in X, \;  \|A x-b\|_Y \leq\e\}}(x),
\end{equation}
and its dual problem writes
\begin{equation}
\label{dual-closed}
\tag{$\mathcal{D}_{\e}^\mathrm{cl}$}
\inf_{y \in Y} J_\e^\mathrm{cl}(y), \qquad J_\e(y) := \tfrac{1}{2}\|(A^\ast y)_+\|^2_X -\scl{b}{y}_Y + \e \|y\|_Y,
\end{equation}
where we recall that $z_+$ means the projection of a vector $z \in X$ onto $P$.

As described above, solving this primal-dual pair of optimisation problems leads to solving \eqref{approx-Farkas-problem} and~\eqref{exact-Farkas-problem}, thanks to~\eqref{why-aux-prob}, and yields constructive characterisations of solutions. Since these auxiliary problems involve elementary functions, the constructive characterisations we obtain are remarkably direct:
\begin{thrm} 
\label{intersect-ball-opt}
Under the hypotheses of Theorem \ref{intersect-ball-ncs}, for any $b\in Y$, 
\begin{itemize}
\item[(i)] If $b \in \overline{A(P)}$, then for any $\e>0$, 
\begin{itemize}
\item$(\mathcal{P}_{\e}^\mathrm{cl})$ admits a unique solution $x_\e^\star \in P$, such that  
$\|A x_\e^\star-b\|_Y \leq \e$.
\item $(\mathcal{D}_{\e}^\mathrm{cl})$ admits a unique solution~$y_\e^\star$, and 
\[x_\e^\star := (A^\ast y_\e^\star)_+.\]
\end{itemize}

\item[(ii)] If $b \in A(P)$, then \begin{itemize}
\item$(\mathcal{P}_{0}^\mathrm{cl})$ admits a unique solution $x^\star\in P$, such that $Ax^\star=b$. \item\textbf{If furthermore $(\mathcal{D}_{0}^\mathrm{cl})$ admits a solution $y^\star$}, then
\[x^\star := (A^\ast y^\star)_+.\]
\end{itemize}

\end{itemize}
\end{thrm}


\cz{
\paragraph{Least-norm (conic) pseudoinverse viewpoint.}
When $\varepsilon=0$ and $b\in A(P)$, problem $(\mathcal P_0^{\mathrm{cl}})$ computes the minimum-norm solution of the conic system $Ax=b$, $x\in P$. In this closed-cone setting, it is hence natural to define the cone-constrained least-norm pseudoinverse $A_P^\dagger(b) : A(P) \to P$ of $A$ by
$$
\forall b \in A(P), \quad 
\{A_P^\dagger(b)\} = \argmin \{ \Vert x\Vert_X,\; Ax=b,\, x\in P \}.
$$
When $P=X$ and $\mathrm{Ran}(A)$ is closed, this coincides with the usual Moore-Penrose pseudoinverse $A^\dagger$ on $\mathrm{Ran}(A)$. Thus, Theorem 1.4 does not merely prove solvability: it canonically selects, among all conic solutions, the least-norm one, and provides the explicit primal-dual representation $A_P^\dagger(b) = (A^*y^\star(b))_+$, whenever a dual minimiser $y^\star(b)$ exists. This also suggests a practical numerical procedure: solve the unconstrained dual problem $(\mathcal D^{\mathrm{cl}}_0)$ and recover the primal variable by projecting $A^\ast y^\star(b)$ onto $P$.

}

\subsubsection{Cone generated by a bounded closed convex set}
In the more general case of a cone $P_r$ generated by a bounded closed convex set $K_r$ containing $0$, the constructive characterisation is, accordingly, more involved. 
\begin{thrm}
\label{convex-general-opt}
Under the hypotheses of Theorem \ref{general-ncs}, for any $b\in Y$, 
\begin{itemize}
\item[(i)] If $b \in \overline{A(P_r)}$, then for any $\e>0$, 
\begin{itemize}
\item$(\mathcal{P}_\e)$ admits a solution $x^\star_\e\in P_r$ such that $\|Ax^\star_\e-b\|_Y \leq \e$, 
\item $(\mathcal{D}_\e)$ admits a unique solution $y_\e^\star$, and
\begin{equation}\label{opt-cond-approx-conv-gen}x_\e^\star \in \sigma_{K_r}(A^\ast y_\e^\star) \, \partial \sigma_{K_r}(A^\ast y_\e^\star).\end{equation}
\end{itemize}
\item[(ii)] If $b \in A(P_r)$, then 
\begin{itemize}
\item$(\mathcal{P}_0)$ admits a solution $x^\star\in P_r$ such that $Ax^\star=b$,
\item if furthermore $(\mathcal{D}_0)$ admits a solution $y^\star$, then 
\begin{equation}\label{opt-cond-exact-conv-gen}x^\star \in \sigma_{K_r}(A^\ast y^\star) \, \partial \sigma_{K_r}(A^\ast y^\star).\end{equation}
\end{itemize}
\end{itemize}
\end{thrm}

\cz{\begin{rmrk}
For $\e>0$, the dual variable $y^\star_\varepsilon$ can be interpreted as a normal-cone Lagrange multiplier associated with the feasibility constraint $Ax\in \overline{B(b,\varepsilon)}$. Indeed,
from the saddle-point conditions \eqref{saddle-point} one has
$y^\star_\varepsilon \in -\partial\delta_{\overline{B(b,\varepsilon)}}(A x^\star_\varepsilon) = -N_{\overline{B(b,\varepsilon)}}(A x^*_\varepsilon)$, where $N_{\overline{B(b,\varepsilon)}}(A x^*_\varepsilon)$ is the normal cone to $\overline{B(b,\varepsilon)}$ at $A x^*_\varepsilon$.
In particular, if the constraint is active (and we will show it necessarily is when $\|b\| > \e$), i.e., $\|A x^\star_\varepsilon-b\|_Y=\varepsilon$, then $ y^\star_\varepsilon = \lambda_\varepsilon\,\frac{b-Ax^\star_\varepsilon}{\|b-Ax^\star_\varepsilon\|_Y}$ for some $\lambda_\varepsilon\geq 0$, which is the usual complementary-slackness relation for a (Hilbertian)-ball constraint.
\end{rmrk}}

Note that the optimality conditions \eqref{opt-cond-approx-conv-gen}--\eqref{opt-cond-exact-conv-gen} are only necessary: if the subdifferentials involved contain more than one vector, there is \textit{a priori} no way to determine which one the actual solution $x_\e^\star$ (or $x^\star$) is. Thus, for the previous result to be more constructive, it is relevant to study tractable sufficient conditions ensuring that these subdifferentials are reduced to singletons (\textit{i.e.,} actually are gradients). In this case, the above inclusion then uniquely determines $x_\e^\star$ (or $x^\star$). 

\cz{
\paragraph{Singular vectors.} Let us introduce the set of \textit{singular vectors} $\operatorname{sing}(K_r)$ of $K_r$, that is, the set of vectors $z \in X$ such that the above subdifferential $\partial \sigma_{K_r}(z)$ is not reduced to a singleton. Equivalently, this corresponds to $\sigma_{K_r}$ failing to be differentiable at $z$. For instance, if $K_r$ is a polytope, singular vectors are exactly those normals that expose a face of dimension $\ge 1$ (an edge, a facet, etc).
}

We are thus led to the hypothesis
\begin{equation}
\tag{$\text{H}$}
\label{extremal}
\forall y \in Y,  \quad  A^\ast y  \in \sing(K_r)^c \cup P_r^\circ,
\end{equation}
which can also be written as $\mathrm{Ran}(A^\ast)\subset \sing(K_r)^c \cup P_r^\circ$.

Assumption~\ref{extremal} can be seen as a non-degeneracy condition ensuring that the support-function subgradient used in \eqref{opt-cond-approx-conv-gen}--\eqref{opt-cond-exact-conv-gen} for the reconstruction of $x_\varepsilon^\star$ is single-valued along $\mathrm{Ran}(A^*)$.

\begin{prpstn}\label{prop-opt-cond-ncs}
    Under \eqref{extremal}, conditions \eqref{opt-cond-approx-conv-gen} and \eqref{opt-cond-exact-conv-gen} become, respectively,
    \begin{equation}
        \label{opt-cond-approx-single}
        \{x_\e^\star\}=\sigma_{K_r}(A^\ast y_\e^\star)\,\partial \sigma_{K_r}(A^\ast y_\e^\star),
    \end{equation}
    and 
    \begin{equation}
        \label{opt-cond-exact-single}
        \{x^\star\}=\sigma_{K_r}(A^\ast y^\star)\,\partial \sigma_{K_r}(A^\ast y^\star).
    \end{equation}
    That is, these optimality conditions are necessary and sufficient.
\end{prpstn}
Of course, assumption~\eqref{extremal} is a much stronger hypothesis than what is actually needed: we only need the inclusion underlying~\eqref{extremal} to hold at a minimiser $y^\star$ of $J_\e$. 

\begin{rmrk}
    A typical example where \eqref{extremal} holds is given in the previous section, with $P_r$ a convex and closed cone. As noted in Remark \ref{rq-Kr-is-the-ball}, the corresponding generator set $K_r$ is the intersection of $P_r$ with the unit ball. In this case, we have $\sing(K_r)\subset P_r^\circ$. Indeed, using Lemmas \ref{intersect-ball} and \ref{lem-charac-normal-cone}, for any $v\in X$, if $v_+\neq 0$, then $\sigma_{K_r}(v)=\|v_+\|_X$ and $\partial \sigma_{K_r}(v)=v_+/\|v_+\|_X$. Hence, $v\in \sing(K_r)$ implies $v_+=0$, \ie $v\in P_r^\circ.$
    
As a result,
    \[\sing(K_r)^c \cup P_r^\circ=X,\]
and \eqref{extremal} holds.

    This illustrates the importance of the choice of a generator for the convex cone, when this choice is available, and a generic example thereof. 
\end{rmrk}
\subsubsection{General cone}
We now consider the case where $P$ might also not be convex: we let $P$ be any cone, and let $K$ be a bounded set containing $0$ that generates $P$, i.e., such that $\cone(K) = P$. We \textbf{relax the constraints} by defining $K_r$ to be the closed convex hull of $K$, namely $K_r:= \adhco(K)$ and $P_r = \cone(K_r)$, a convex cone that is not necessarily closed. The diagram below summarises the overall relaxation approach.

\begin{equation*} 
         \begin{tikzcd}[column sep=huge, row sep = huge]
P  \arrow{r}{\subset} & P_r  \\%
K \arrow{r}{\adhco} \arrow{u}{\cone}& K_r  \arrow{u}{\cone}
\end{tikzcd}
\label{diagram-relax}
\end{equation*}

We now apply the previous approach to the cone $P_r$ and corresponding generating set $K_r$.
The key question is now whether the constructive characterisations \eqref{opt-cond-approx-conv-gen}--\eqref{opt-cond-exact-conv-gen} lead to elements in the original constraint cone~$P$, and not merely in $P_r$. 
This motivates the following assumption.
\begin{equation}
    \label{extremality}
    \tag{E}
    \ext(K_r) \subset K,
\end{equation}
where $\ext(K_r)$ refers to extreme vectors of $K_r$, i.e., that vectors of $K_r$ that cannot be written as a nontrivial convex combination of vectors in $K_r$.

Assumption~\eqref{extremality} will often be satisfied in practice. Indeed, since $K$ is assumed to be bounded, so is the closed convex set $K_r$; hence the latter set is weakly compact. As a result, an application of Milman's theorem (see~\cite{Rudin}[Theorem 3.24]) with either the strong or the weak topology shows that~\eqref{extremality} holds true as soon as either one of the following hypotheses holds
\begin{itemize}
    \item $K$ is weakly closed,
    \item $K_r$ is strongly compact and $K$ is strongly closed.
\end{itemize}

\begin{thrm}
\label{general-opt}
There hold
\begin{itemize}
\item[(i)]
If $b \in \overline{A(P_r)}$, then for any $\e>0$,
\begin{itemize}
\item $(\mathcal{P}_\e)$ admits a solution $x_\e^\star \in P_r$ such that $\|A x_\e^\star -b\|_Y\leq \e$, 
\item $(\mathcal{D}_\e)$ admits a unique solution~$y_\e^\star$,
\item under~\eqref{extremal}-\eqref{extremality}, $x^\star_\e$ is the unique solution to $(\mathcal{P}_\e)$, $x_\e^\star \in P$ and \[\{x_\e^\star\} = \sigma_{K_r}(A^\ast y_\e^\star) \, \partial \sigma_{K_r}(A^\ast y_\e^\star).\] 
\end{itemize}

\item[(ii)] If $b \in A(P_r)$, 
\begin{itemize}
    \item  $(\mathcal{P}_0)$ admits a solution $x^\star \in P_r$ such that $A x^\star = b$, 
    \item if furthermore $(\mathcal{D}_0)$ admits a solution~$y^\star$, then under~\eqref{extremal}-\eqref{extremality}, $x^\star \in P$ and \[\{x^\star\} = \sigma_{K_r}(A^\ast y^\star) \, \partial \sigma_{K_r}(A^\ast y^\star).\]   
\end{itemize}

\end{itemize}
\end{thrm}

Notice that in the above theorem, $(i)$ implies that, as soon as $b\in \overline{A(P_r)}$, then $b \in \overline{A(P)}$. Indeed, the existence of solutions to the auxiliary problem \eqref{primal-general} and its dual \eqref{dual-general} holds independently of $b\in Y$. We have thus obtained the following corollary: 
\begin{crllr}
Under \eqref{extremal}-\eqref{extremality}, there holds  $\overline{A(P_r)}= \overline{A(P)}$.
\end{crllr}
\begin{rmrk}
    This type of statement is common in optimal control theory, in the form of so-called \textit{bang-bang principles}: in many optimal control problems, the set of reachable states using only controls which saturate the constraints (bang-bang controls), is equal to the set of reachable states using all admissible controls~\cite{HermesLasalle1969}. Controls which saturate the constraints are typically extreme points of the control constraint set.
\end{rmrk}

\subsection{Remarks and perspectives}
\paragraph{Uniqueness.} In all our results regarding the approximate solvability problem, the minimiser of the dual functional is unique. In contrast, that is no longer the case for the exact solvability problem, even in simple cases.

Consider $X = Y = \R^2$, $P_r = \R_+^2$, $K_r = P_r \cap \overline{B(0,1)}$, $A = \mathrm{Id}$, $b = (1,0)$. Clearly, $b \in A(P_r)$.
The dual functional $J_0$ then writes
\[J_0(y) = J_0(y_1,y_2)=\tfrac{1}{2} \|y_+\|_2^2 - y_1.  \]
One easily checks that $J$ reaches its infimum value $-\tfrac{1}{2}$ on the whole half-line $\{(1,a), \ a\leq 0\}$. Note however that the constrained linear equation $Ax=b$ with $x\in P$ has a unique solution.

\paragraph{Generalisations.}
The Hilbert space setting we have chosen is sufficiently general to be suitable for applications requiring infinite-dimensional spaces (e.g., linear control theory), yet sufficiently simple to be tractable. 

We are however fairly confident that the core of our results can be generalised to reflexive Banach spaces, or at least to uniformly convex Banach spaces, upon using proper generalisations of the relevant tools (notably, Moreau's decomposition in Banach spaces~\cite{combettes2013moreau}).

\paragraph{Constructiveness.} In the second parts of Theorems \ref{intersect-ball-opt}, \ref{convex-general-opt}, \ref{general-opt}, the dual problem for the exact solvability problem may or may not admit a minimiser. This issue, which plays a central role in the constructiveness of our method, is discussed in further detail in Section \ref{sec-constructiveness}, where we introduce the notion of \textit{weak constructiveness}. We will focus on the existence of an identity characterising the solution of the primal problem (when it exists), notwithstanding hypotheses of the type \eqref{extremal}--\eqref{extremality} guaranteeing that this identity characterises $x^\star$ uniquely. We obtain some partial results and counterexamples which raise important questions on the general issue of constructive characterisations for solutions to constrained equations.

\paragraph{Intrinsic condition for the exact solvability problem.} In the case of a closed cone, the results we provide in Theorems \ref{intersect-ball-ncs} and \ref{intersect-ball-opt} are intrinsic to the cone. This is also the case for the necessary and sufficient condition we provide in Theorem \ref{general-ncs} for the approximate solvability problem.

However, the rest of our results seem to depend strongly on the generator $K_r$ of our cone. It would be interesting to obtain an intrinsic version of these results, \textit{e.g.} an analog to the identity $x^\star=(A^\ast y^\star)_+$,
which would hold for a convex (and not necessarily closed) cone.

\paragraph{Class of cones under consideration.}
In the general case, we consider convex cones that are generated by a bounded closed convex set containing $0$. It would be relevant to describe this class of cones in a less indirect way. Going further, and considering a somewhat dual question, can we describe the set of bounded closed convex generators for such a cone? 

This is related to the previous paragraph concerning a possible intrinsic characterisation of exact solvability. Indeed, the necessary and sufficient condition for exact solvability, and the constructive characterisation of the solution $x^\star$, seem to depend on the generator $K$ of the cone. Having a more detailed description of the class of possible generators, and more knowledge on the properties of the cone $P$, could allow us to consider either an ``optimal'' generator to write these results, or a ``generic'' one. For example, in the case where the cone is closed, the generic generator seems to be the intersection of the cone with the closed unit ball.

\paragraph{Numerical approximations of solutions.} 
Because of its constructive nature and its reliance upon the Fenchel-Rockafellar Theorem, our approach lends itself to using algorithms that find saddle-points of the corresponding Lagrangian, such as the Chambolle-Pock algorithm~\cite{Chambolle_2011}. 

For applications to control theory or inverse problems associated to ODEs and/or PDEs, at least one of the Hilbert spaces is infinite-dimensional. This poses the question of how to appropriately discretise the involved functionals, especially in view of preserving the primal-dual structure at the discrete level.

\section{Mathematical framework}
\subsection{Notations}
\label{subsec-notations}
For the readers' convenience, let us define the convex analytic tools used in the present work, in some generic (real) Hilbert space $H$.
\subsubsection{Functions}
Let $f : H \to \R \cup \{+\infty\}$.

We let $\dom(f) := \{x \in H, \, f(x) < +\infty\}$ be its \textbf{domain}; $f$ is said to be \textbf{proper} when $\dom(f) \neq \emptyset$. We define its \textbf{convex conjugate} by 
\[\forall x \in H, \quad f^\ast(x) := \sup_{y \in H}\; (\langle x,y\rangle - f(y)).\]

We let $\Gamma_0(H)$ be the set of all proper, convex and lower semicontinuous functions on $H$ and recall Fenchel-Moreau's identity: if $f \in \Gamma_0(H)$ then $f^\ast \in \Gamma_0(H)$ and $(f^\ast)^\ast = f$.

Finally, for a given $f\in \Gamma_0(H)$, and for $x \in H$, we define the \textbf{(convex) subdifferential} of $f$ at $x$ through 
\[\partial f(x) := \{p \in H, \; \forall y \in H, \quad  f(y) \geq f(x) + \langle p,y-x\rangle\}.\]

\subsubsection{Convex sets}
Let $C \subset H$ be a convex set.

We let $\delta_C$ be its \textbf{indicator} function, namely $\delta_C(x) = 0$ if $x\in C$ and $+\infty$ otherwise, and $\sigma_C = \delta_C^\ast$ be its \textbf{support} function; in other words, $\sigma_C(x) = \sup_{y \in C} \langle x, y \rangle$. 
We also define its \textbf{gauge} function $j_C$, defined for $x \in H$ by $j_C(x) = \inf \{\alpha > 0, \, x \in \alpha C\}$. 

The set $\ext(C)$ of \textbf{extreme} vectors of $C$  refers to vectors $x \in C$ that cannot be written as a non-trival convex combination of two vectors in $C$.

If $C$ is closed and bounded, $\sigma_C(x) <+\infty$ for all $x \in C$ and the supremum defining $\sigma_C(x)$ is a maximum. In this case, recall that $\sing(C)$ denotes the set of \textbf{singular vectors} to $C$, namely those vectors $x \in H$ for which the set of maximisers associated to the maximum defining $\sigma_C(x)$ is not reduced to a singleton. 

Finally, for $x \in H$, we let $N_C(x) = \{p \in H, \; \forall y  \in C, \, \langle p,x \rangle \geq \langle p,y \rangle\}$ be the \textbf{normal cone} to $C$ at $x$.

For an arbitrary set $K \subset H$, we let $\adhco(K)$ be its \textbf{closed convex hull},  namely the smallest closed and convex set containing $K$. 

\subsubsection{Cones}

Let $P \subset H$ be a cone.

We let $P^\circ = \{x \in H, \; \forall y \in P, \, \langle x, y\rangle \leq 0\} $ be the \textbf{polar cone} of $P$. If $P$ is closed and convex, given $x \in H$ we let $x_+$ be its projection onto $P$ and $x_-$ be its projection onto $P^\circ$. Moreau's decomposition theorem yields $x = x_+ + x_-$ with $\scl{x_+}{x_-} = 0$. 

Finally, we highlight two useful identities (see \cite[Proposition 6.26]{bauschke-combettes} and \cite[Proposition 6.34]{bauschke-combettes} respectively): for two nonempty convex cones $P_1$ and $P_2$, there holds
    \begin{equation}
\label{identities-cones}
        P_1^\circ \cap P_2^\circ =(P_1+P_2)^\circ, \qquad (\overline{P_1}\cap \overline{P_2})^\circ=\overline{P_1^\circ+P_2^\circ}.
    \end{equation}

\subsection{Fenchel-Rockafellar duality and constructive characterisation of solutions}
Let $P_r = \cone(K_r)$ be a cone, where $K_r$ is a closed bounded convex set containing $0$. Hence $P_r$ is a convex cone containing $0$, and is not necessarily closed. It is straightforward but essential to note the identity 
\begin{equation}
    \label{dom-gauge}
    \dom(j_{K_r}) = P_r.
\end{equation}

The purpose of this section is to define the (primal) optimisation problem underlying our results, derive its dual problem, establish strong duality and make the corresponding saddle-point conditions explicit.  
\subsubsection{Primal and dual optimisation problems}

For any $\e \geq 0$, we consider the \textit{primal optimisation problem}, 
\begin{align}
\label{primal-fct}
    \pi_\e = \inf_{x \in X}  \tfrac{1}{2} j_{K_r}^2(x) + \delta_{\{x \in X,\; \|A x -b\|_Y \leq \e\}}(x)
     = \inf_{x \in X} \tfrac{1}{2} j_{K_r}^2(x) + \delta_{\overline{B(b,\e)}}(Ax).
\end{align}
For a given $\e \geq 0$ and in view of~\eqref{dom-gauge}, note that the functional underlying the primal  problem is not identically equal to $+\infty$ if and only if there exists $x \in X$ such that $x \in P_r$ and $\|A x -b\|_Y \leq \e$. In other words, $\pi_\e<+\infty$ if and only if there exists $x \in X$ such that $x \in P_r$ and $\|A x -b\|_Y \leq \e$. Consequently, the case $\e=0$ amounts to exact solvability. 
\begin{prpstn}
\label{strong-duality}
For any $\e \geq 0$, the following strong duality holds 
\begin{equation}\label{eq-strong-duality}
\pi_\e = \inf_{x \in X} \tfrac{1}{2} j_{K_r}^2(x) + \delta_{\overline{B(b,\e)}}(Ax) = - \inf_{y \in Y} J_\e(y),\end{equation}
where the dual functional $J_\e$ is given by
\[\forall y \in Y, \quad J_\e(y) := \tfrac{1}{2}\sigma_{K_r}^2(A^\ast y)  - \langle b, y\rangle_Y + \e \|y\|_Y. \]
Furthermore, if $\pi_\e$ is finite, it is attained. 
\end{prpstn}
Since $K_r$ is bounded, $\sigma_{K_r}$ is Lipschitz continuous. The function $J_\e$ is thus continuous and, in particular, takes finite values on the whole of $Y$.
\begin{proof}
The primal functional can be written in the form $F(x) +G(Ax)$, where $F = \tfrac{1}{2} j_{K_r}^2$, $G = \delta_{\overline{B(b,\e)}}$. Clearly, $G \in \Gamma_0(Y)$, while by Lemma~\ref{conj-supp-square}~(ii), $F \in \Gamma_0(X)$.
Its dual functional, defined on $Y$, is then given by
\[y \mapsto F^\ast(A^\ast y) + G^\ast(-y) \]
It is standard that $G^\ast = \delta_{\overline{B(b,\e)}}^\ast = \langle b, \cdot\rangle_Y + \e \|\cdot\|_Y$, and Lemma~\ref{conj-supp-square}~(ii) shows that $F^\ast = (\tfrac{1}{2} j_{K_r}^2)^\ast = \tfrac{1}{2} \sigma_{K_r}^2$.

We now apply the Fenchel-Rockafellar Theorem~\cite{Rockafellar1967}: if there exists $y \in Y$ such that $G^\ast(-y) < +\infty$ and $F^\ast$ is continuous at $A^\ast y$, then strong duality holds and the infimum of the primal optimisation problem is attained if finite.
Here, $G^\ast$ takes finite values on $Y$, while $F^\ast$ is continuous on $X$ as already mentioned. Hence any $y \in Y$ can be used to apply the Fenchel-Rockafellar theorem, and the proof is finished.  
\end{proof}

\subsubsection{Optimality conditions}
\label{subsubsec-opt-cond}
What follows is standard folklore related to (strong) duality associated to the Fenchel-Rockafellar Theorem~\cite{Rockafellar1967}.

A couple $(x^\star, y^\star) \in X \times Y$ is a primal-dual optimal pair (i.e., $x^\star$ minimises the primal functional, and $y^\star$ minimises the dual functional) if and only if $(x^\star, -y^\star)$ is a saddle-point of the Lagrangian 
\[(x,z) \in X \times Y \mapsto \langle z, A x\rangle_Y + F(x) - G^\ast(z) = \langle z, A x\rangle_Y + \tfrac{1}{2} j_{K_r}^2(x) -   \langle b, z\rangle_Y - \e \|z\|_Y. \]
This is also equivalent to the two following first order optimality conditions
\begin{equation}
\label{saddle-point}
\begin{cases}
    x^\star \in \partial F^\ast(A^\ast y^\star)  = \partial (\tfrac{1}{2}\sigma_{K_r}^2)(A^\ast y^\star) = \sigma_{K_r}(A^\ast y^\star)\, \partial \, \sigma_{K_r}(A^\ast y^\star) \\
     y^\star \in -\partial G(A x^\star)     
    \end{cases},
\end{equation}
where the identity used above in the first line is due to the chain rule for subdifferentials, see for instance~\cite{clarke1990optimization}[Theorem 2.3.9, (ii)].

By the results and the remarks of the previous subsection, given that $\pi_\e = -\inf_{y \in Y} J_\e(y)$ for all $\e \geq 0$, we have
\begin{equation}\label{dual-characterisations}
\begin{aligned}
\forall \e > 0, \; \inf_{y \in Y} J_\e(y) > -\infty & \quad  \implies \quad  b\in \overline{A( P_r)}, \\
  \inf_{y \in Y} J_0(y) > -\infty & \quad \implies  \quad b\in A( P_r). 
\end{aligned}
\end{equation}

Combining \eqref{dual-characterisations} with the optimality conditions \eqref{saddle-point}, we immediately obtain:
\begin{lmm}\label{opt-cond-rmrk}
    Let $\e \geq 0$. If $J_\e$ admits a minimiser $y_\e^\star$, then there exists 
\begin{equation*} x_\e^\star \in \sigma_{K_r}(A^\ast y_\e^\star) \, \partial \sigma_{K_r}(A^\ast y_\e^\star),\end{equation*}
which is a minimiser for \eqref{primal-general}. In particular, $x_\e^\star \in P_r$ and $\|A x_\e^\star-b\|_Y \leq \e$.
\end{lmm}

\begin{crllr}\label{cor-single-x-in-subdiff}
    If $(x^\star, y^\star)$ is a primal-dual optimal pair such that $\sigma_{K_r}(A^\ast y^\star) \, \partial \sigma_{K_r}(A^\ast y^\star)$ is reduced to a singleton, then this singleton equals $x^\star$. 
\end{crllr}

The proofs of our main results will rely greatly on the above observations,and a careful examination of the relationship between the Farkas problems \eqref{approx-Farkas-problem} and \eqref{exact-Farkas-problem}, and the optimisation problems \eqref{primal-general} and \eqref{dual-general} for $\e\geq 0$.

\section{Proofs}
\paragraph{Outline of the proofs.} 
 We first focus on cones generated by a bounded closed convex set containing $0$, \ie Theorems \ref{general-ncs} and \ref{convex-general-opt}.

We tackle approximate solvability, then exact solvability.
Section \ref{subsec-approx-proof} provides:
\begin{itemize}
    \item proof of Theorem \ref{general-ncs} $(i)$ - necessary and sufficient conditions for approximate solvability;
    \item proof of the uniqueness of dual minimisers (Lemma \ref{unique-approx});
    \item proof of Theorem \ref{convex-general-opt} $(i)$ - constructive characterisation and uniqueness of solutions in the approximate case.
\end{itemize}

In Section \ref{subsec-exact-proof}:
\begin{itemize}
    \item we prove Theorem \ref{general-ncs} $(ii)$ - a necessary and sufficient condition for exact solvability;
    \item we deduce Theorem \ref{convex-general-opt} $(ii)$ regarding the characterisation of solutions when they exist.
\end{itemize}

Section \ref{subsec-closed-proof} derives Theorems \ref{intersect-ball-ncs} and \ref{intersect-ball-opt}, which are particular instances of the above theorems, with a closed cone.

Finally, in Section \ref{subsec-constructivity-uniqueness}, we focus on situations where the solution can be completely characterised:
 \begin{itemize}
     \item in Subsection \ref{subsubsec-convex-constructive} we address the case of convex cones, proving Proposition \ref{prop-opt-cond-ncs};
     \item in Subsection \ref{subsubsec-general-constructive} we turn to general cones using relaxation, and prove Theorem \ref{general-opt}.
 \end{itemize}
Within this section, we shall often use the trivial fact that $\sigma_{K_r}$ is $1$-positively homogeneous.

\subsection{Approximate solvability}

\label{subsec-approx-proof}
\subsubsection{Necessary and sufficient condition for solvability}
\label{subsubsec-approx-ncs}
We begin by proving Theorem \ref{general-ncs}, $(i)$. First, we check that:
\begin{lmm}\label{lem-separation-interpretation}
   For $y\in Y$, $\sigma_{K_r}(A^\ast y) = 0$ if and only if $A^\ast y \in P_r^\circ$.
\end{lmm}
\begin{proof}
        Let us first remark that $\sigma_{K_r}(A^\ast y) = 0$ if and only if $\langle A^\ast y, v \rangle_X \leq 0$ for all $v \in K_r$. Given that $P_r = \cone(K_r)$, it is equivalent to $\langle A^\ast y, v \rangle_X \leq 0$ for all $v \in P_r$, \ie $A^\ast y \in P_r^\circ$. 
\end{proof}
We now have two equivalent formulations of the solvability condition in Theorem \ref{general-ncs}, $(i)$, one that is intrinsic to $P_r$, and one that depends on the choice of generator $K_r$. We now proceed with the proof, using the latter.
\begin{lmm}\label{lem-coercivity-interpretation}
    The following holds:
\[\Big(\forall y \in Y, \; \sigma_{K_r}(A^\ast y) = 0 \, \implies \,\scl{b}{y}_Y \leq 0 \Big)\quad \implies \quad  \forall \e>0,\quad\inf_{y \in Y} {J_\e}(y) = \min_{y \in Y} {J_\e}(y) > -\infty.\]    
\end{lmm}
\begin{proof}
Assume that, for all $y \in Y$, $\sigma_{K_r}(A^\ast y) = 0 \, \implies \,\scl{b}{y}_Y \leq 0$.

Since the functionals $J_{\e}, \e > 0$  are in $\Gamma_0(Y)$, the result will be proved if we establish that $J_{\e}$ is coercive for all $\e>0$. 
    
    We let $\e>0$ be fixed.
    By contradiction, if $J_\e$ is not coercive there exists a sequence $(y_n) \in Y^\N$ such that $\|y_n\|_X\to \infty$ with $(J(y_n))$ upper bounded, say $J(y_n) \leq M$ for all $n \in \N$. Denoting $z_n := \tfrac{y_n}{\|y_n\|_Y}$ and since $\sigma_{K_r}^2$ is $2$-positively homogeneous, we have
    \begin{equation*}
        \tfrac{1}{2} \|y_n\|_Y^2 \,\sigma_{K_r}^2(A^\ast z_n)-\|y_n\|_Y \langle b, z_n\rangle_{Y} + \e\|y_n\|_Y  = J_\e(y_n) \leq M.
    \end{equation*}
    Consequently, if $\liminf_{n\to \infty} \sigma_{K_r}^2(A^\ast z_n)>0$, then the left-hand side above diverges to $+\infty$ which is not possible.
    Let us now treat the remaining case where $\liminf_{n\to \infty} \sigma_{K_r}^2(A^\ast z_n)=0$. We may upon extraction of a subsequence assume both that $\lim_{n\to \infty} \sigma_{K_r}^2(A^\ast z_n)=0$, and owing to the boundedness of $(z_n)$, that $z_n \rightharpoonup z$ weakly in $Y$ for some $z \in Y$. Since $A^\ast  \in L(Y,X)$, we have $A^\ast z_n \rightharpoonup A^\ast z$ weakly in $X$.

    The function $\sigma_{K_r}^2$ is convex and strongly lower semicontinuous on $E$, hence it is (sequentially) weakly lower semicontinuous, which implies that $\sigma_{K_r}^2(A^\ast z) \leq \lim_{n\to \infty} \sigma_{K_r}^2(A^\ast z_n)=0$ and enforces $\sigma_{K_r}^2(A^\ast z)  = 0$. By our assumption, this implies $\langle b, z\rangle_Y \leq 0$. In this case, we have
    \begin{equation*}
        \|y_n\|_Y \, (-\langle b, z_n\rangle_{Y} +\e)   \leq J_\e(y_n) \leq M,
    \end{equation*}
    and at the limit we again reach a contradiction.
\end{proof}

Let us now prove Theorem~\ref{general-ncs} (i), namely the following Proposition.
\begin{prpstn}\label{ncs-approx} We have the equivalence
\[b \in \overline{A( P_r)} \quad \iff \quad \forall y \in Y, \; \sigma_{K_r}(A^\ast y) = 0 \, \implies \,\scl{b}{y}_Y \leq 0. \]\end{prpstn}
\begin{proof}
    Assume that $b \in \overline{A(P_r)}$: we may find $(w_n) \in P_r^\N$ such that the sequence $(A w_n)$ converges to~$b$. By contradiction, and using Lemma \ref{lem-separation-interpretation}, assume that there exists $y \in Y$ such that $A^\ast y \in P_r^\circ$ and $\scl{b}{y}_Y > 0$. Then we should have \[0 \geq \langle A^\ast y, w_n\rangle_X = \langle y, A w_n \rangle_Y \longrightarrow \langle y,b\rangle_Y > 0, \]
    which is impossible.

    We now tackle the converse assertion, assuming that, for all $y \in Y$, $\sigma_{K_r}(A^\ast y) = 0  \implies \scl{b}{y}_Y \leq 0$.
    By Lemma \ref{lem-coercivity-interpretation}, this implies that $\inf_{y \in Y} J_\e(y) > -\infty$ for all $\e>0$.  Combined with \eqref{dual-characterisations}, this concludes the proof that $b\in \overline{A(P_r)}$.
    


\end{proof}
We can sum up the relationship of the auxiliary optimisation problems \eqref{primal-general}--\eqref{dual-general} to the approximate Farkas problem in this way:
\begin{crllr}\label{cor-equivalence-optim-approx}
    The following propositions are equivalent:
    \begin{itemize}
        \item[(i)] $b\in \overline{A(P_r)}$;
        \item[(ii)] For all $\e>0$, $\inf_{y\in Y} J_\e(y)>-\infty$;
        \item[(iii)] For all $\e>0$, \eqref{primal-general} admits a solution $x^\star_\e.$ 
    \end{itemize}
\end{crllr}
\begin{proof}
    The implication $(i) \implies (ii)$ is a consequence of Proposition \ref{ncs-approx} and Lemma \ref{lem-coercivity-interpretation}.

    The implication $(ii)\implies (iii)$ is a direct consequence of the strong duality given in Proposition \ref{strong-duality}: $\pi_\e=-\inf_Y J_\e<+\infty$ and is hence attained at a minimiser $x^\star_\e$.

    Finally, recalling the implication \eqref{why-aux-prob} given in the introduction, $(iii)$ implies that for all $\e>0$ there exists $x^\star_\e \in P_r$ such that $\|Ax^\star_\e-b\|_Y\leq \e$, \ie $b\in \overline{A(P_r)}.$
\end{proof}

\subsubsection{Constructive characterisation and uniqueness}
\label{subsubsec-approx-cons}
We now turn to the proof of Theorem \ref{convex-general-opt} $(i)$. The previous section already provides the existence of a solution $x^\star_\e$ to \eqref{primal-general}, and a minimiser for the dual functional $J_\e$, under the assumption that $b \in \overline{A(P_r)}$.

We now prove that this dual minimiser is always unique:
\begin{prpstn}
\label{unique-approx}
    If $b \in \overline{A(P_r)}$, then $J_\e$ admits a unique minimiser for all $\e>0$.
\end{prpstn}
\begin{proof}

  Since $b \in \overline{A(P_r)}$, Lemma~\ref{lem-coercivity-interpretation} shows that for all $\e>0$, $J_\e$ has a finite infimum, which is in fact a minimum. 

    Let $\e>0$ be fixed, and 
denote by $X^\star$ the set of minimisers for the primal problem \eqref{primal-general}, and by $Y^\star$ the set of minimisers for the dual problem \eqref{dual-general}. Let $(x^\star, y^\star) \in X^\star \times Y^\star$ be a primal-dual optimal pair, which exists by Corollary \ref{cor-equivalence-optim-approx}.

If $\|b\|_Y \leq \e$, $x^\star = 0$ is an optimal primal variable since $\|A 0 -b\|_Y = \|b\|_Y \leq \e$. But $j_{K_r}(x) = 0$ if and only if $x=0$ by Lemma~\ref{basic-gauge}(i), hence $0$ is the unique optimal primal variable. In particular, $AX^\star=\{0\}$.

If $\|b\|_Y > \e$, $0\notin X^\star$ so $\pi_\e$>0. By \eqref{eq-strong-duality}, $\min J_\e <0$ and thus $0\notin Y^\star$. Then, the relation $y^\star \in -\partial \delta_{\overline{B(b, \e)}}(A x^\star)$ of~\eqref{saddle-point}, implies that $A x^\star$ must lie at the boundary of $\overline{B(b, \e)}$. Since $X^\star$ is convex, so is $AX^\star$. Hence, $AX^\star$ is a convex subset of the boundary of $\overline{B(b, \e)}$, \ie the sphere $\{y \in Y, \,\|y-b\|_Y = \e\}$. The closed ball being strictly convex in the Hilbert space $Y$, any convex subset of its boundary must be a singleton. Thus, there exists some $b^\star \in \overline{B(b, \e)}$ with $\|b^\star - b\|_Y = \e$ such that \begin{equation}\label{unique-target}
AX^\star = \{b^\star\}.\end{equation}

Thus, in any case, the set $A X^\star$ is always reduced to a single point, which we can denote $b^\star$ in general.

We now return to the inclusion $y^\star \in - \partial \delta_{\overline{B(b, \e)}}(A x^\star)  =-\partial \delta_{\overline{B(b, \e)}}(b^\star)$.
First assume $\|b\|_Y \leq \e$, then  $b^\star = 0$ and $y^\star \in -\partial\delta_{\overline{B(b, \e)}}(0)$.
If $\|b\|_{Y}<\e$, $0\in B(b, \e)$, and thus the latter subdifferential reduces to $\{0\}$, leading to $y^\star = 0$.
Otherwise, $\|b\|_{Y}=\e$, $0\in \partial \overline{B (b, \e)}$ and we find
\[y^\star \in\left\{\lambda \frac{b}{\e},\, \lambda \geq 0\right\} = \{\lambda b,\, \lambda \geq 0\}.\]
Restricting the function $J_{\e}$ to the above half-line, using the homogeneities of each of its terms, and the fact that $\|b\|_Y=\e$, we get
\begin{equation*}\label{non-H-case-2}\gamma_0(\lambda):=J_{\e}(\lambda b)=a_0 \lambda^2, \quad \lambda\geq 0,\end{equation*}
where $a_0 = \tfrac{1}{2} \sigma_{K_r}^2(A^\ast b)$.
It is clear that $0$ is the unique minimiser of $\gamma_0$. In other words, we have proved $y^\star = 0$ whenever $\|b\|_Y \leq \e$.

Now assume $\|b\|_Y > \e$. In this case, note that $y^\star \neq 0$, since otherwise the dual problem would admit the minimum $0$, hence the primal optimisation problem would also be of minimum $0$. If so, $x = 0$ is an optimal primal variable, which would mean that $\|A 0-b\|_Y = \|b\|_Y \leq \e$, a contradiction.   

To prove the uniqueness of $y^\star \neq 0$, we argue as follows.
Since $b^\star$ lies at the boundary of $\overline{B(b, \e)}$, we have
\[y^\star \in\left\{\lambda \left(\frac{b-b^\star}{\e}\right), \lambda \geq 0\right\} = \{\lambda \left(b - b^\star\right), \lambda \geq 0\}.\]
Restricting $J_{\e}$ to the above half-line as previously, we find
\[\gamma(\lambda):=J_{\e}(\lambda (b-b^\star)) = a_1 \lambda ^2 + a_2 \lambda, \quad \lambda \geq 0,\]
where, using $\|b - b^\star\|_Y = \e$ and the homogeneities involved $
a_1= \tfrac{1}{2} \sigma_{K_r}^2(A^\ast (b-b^\star))$ and $a_2=-\langle b, b-b^\star  \rangle_Y + \e^2$. 
Because $J_\e$ is bounded below and has negative minimum, we must have $a_1 >0$ and $a_2<0$.

Consequently, $\gamma$ has a unique minimiser $\lambda^\star:=-\tfrac{a_2}{2 a_1}>0$. Hence, $y^\star=\lambda^\star(b-b^\star)$,
and the dual optimal variable is unique.
\end{proof}

Combined with Lemma \ref{opt-cond-rmrk}, this proves Theorem \ref{convex-general-opt} $(i)$ regarding the characterisation of minimisers for \eqref{primal-general}.

\subsection{Exact solvability}
\label{subsec-exact-proof}
Since we are interested in exact solvability, we set $\e = 0$ and will be using the corresponding dual functional~$J_0$. We here prove Theorem~\ref{general-ncs}~$(ii)$, and then Theorem~\ref{convex-general-opt}~$(ii)$.

\begin{lmm}\label{lem-continuation-interpretation}
    The following holds:
\[\exists C >0, \; \forall y \in Y, \;  \scl{b}{y}_Y \leq C \,\sigma_{K_r}(A^\ast y) \quad \implies \quad   \inf_{y \in Y} {J_0}(y)  > -\infty.\]    
\end{lmm}
\begin{proof}
A lower bound for~$J_0$ is readily derived:
\[\forall y \in Y, \quad J_0(y) = \tfrac{1}{2} \sigma_{K_r}^2(A^\ast y) - \scl{b}{y}_Y \geq \tfrac{1}{2} \sigma_{K_r}^2(A^\ast y) -  C\sigma_{K_r}(A^\ast y).\]
Since the function $z \mapsto \tfrac{1}{2}z^2 - z$ is lower bounded on $\R$, so is $J_0$.
\end{proof}

We can now prove Theorem \ref{general-ncs} $(ii)$, which we restate in the following proposition:
\begin{prpstn}
\label{ncs-exact}

We have the equivalence
\[b \in A( P_r) \quad \iff \quad \exists C >0, \; \forall y \in Y, \;  \scl{b}{y}_Y \leq C \,\sigma_{K_r}(A^\ast y). \]
\cz{If $b \in A(P_r)$, the best constant $C$ is given by $C = \sqrt{2 \pi_0}$.}
\end{prpstn}
\begin{proof}
Assume that $b \in A(P_r)$. Let $x \in P_r$ be such that $Ax = b$. Since $x \in P_r$, we may write $x = j_{K_r}(x) \, v$ with $v \in K_r$, by Lemma~\ref{basic-gauge}(iv).
For all $y \in Y$, we write
\[\scl{b}{y}_Y = \scl{A x}{y}_Y = \scl{x}{A^\ast y}_Y = j_{K_r}(x) \scl{v}{A^\ast y} \leq C \,\sigma_{K_r}(A^\ast y),\]
with $C = j_{K_r}(x)$. By minimising over all possible $x \in P_r$ such that $A x =b$, we may in fact set $C =  \sqrt{2 \pi_0}$.

We now assume that the inequality holds. 
In this case, Lemma~\ref{lem-continuation-interpretation} shows that $J_0$ has a finite infimum, and we may conclude that $b \in A(P_r)$ by~\eqref{dual-characterisations}.
\end{proof}

We can now state the analog of Corollary \ref{cor-equivalence-optim-approx} for exact solvability problems:
\begin{crllr}\label{cor-equivalence-optim-exact}
    The following propositions are equivalent:
    \begin{itemize}
        \item[(i)] $b\in A(P_r)$;
        \item[(ii)] $\inf_{y\in Y} J_0(y)>-\infty$;
        \item[(iii)] The primal problem $(\mathcal{P}_0)$ admits a solution $x^\star.$ 
    \end{itemize}
\end{crllr}
\begin{proof}
    The implication $(i) \implies (ii)$ is a consequence of Proposition \ref{ncs-exact} and Lemma \ref{lem-continuation-interpretation}.

    The implication $(ii)\implies (iii)$ is a direct consequence of the strong duality given in Proposition \ref{strong-duality}: $\pi_0=-\inf_Y J_0<+\infty$ and is hence attained at a minimiser $x^\star$.

    Finally, recalling the implication \eqref{why-aux-prob} given in the introduction, $(iii)$ implies that $Ax^\star=b$, \ie $b\in A(P_r).$
\end{proof}

To prove Theorem \ref{convex-general-opt} $(ii)$, we now let $b\in A(P_r)$.
Note that, in contrast to approximate solvability problems, the associated dual functional $J_0$ does not always admit a minimiser, even though its infimum is finite. As a result, even though a minimiser $x^\star$ always exists for $(\mathcal{P}_0)$, the constructive characterisation \eqref{opt-cond-exact-conv-gen} may not be available. 
However, under the hypothesis that $J_0$ does admit a minimiser, we can apply Lemma~\ref{opt-cond-rmrk} with $\e=0$ and conclude.

\subsection{Closed convex case}
\label{subsec-closed-proof}

We now turn to the particular case of closed convex cones, described by Theorems \ref{intersect-ball-ncs} and \ref{intersect-ball-opt}. Let $P_r$ be a closed and convex cone containing $0$. Then, noting that $\cone( P_r \cap \overline{B(0,1)}) = P_r$, we have a natural generating set $P_r \cap \overline{B(0,1)}$ which is a closed, convex and bounded set containing $0$. 

\begin{lmm}
\label{intersect-ball}
Let $K_r := P_r \cap \overline{B(0,1)}$. Then for all $x \in X$, there holds
\[\sigma_{K_r}(x) = \|x_+\|_X \quad \text{ and }\quad \sigma_{K_r}(x)\, \partial (\sigma_{K_r})(x) =  \{x_+\}.\]  
\end{lmm}
\begin{proof}
    Let us start with the case where $x_+ = 0$. In this case $x \in P_r^\circ$, hence $\sigma_{K_r}(x) \leq 0$, and since $K_r$ contains $0$, we have $\sigma_{K_r}(x) = 0 = \|x_+\|_X$. Hence, both identities hold.
    
    Now assume that $x_+ \neq 0$. For all $v \in K_r$, i.e., $v \in P_r$ with $\|v\|_X \leq 1$, 
\[\scl{v}{x}_X = \scl{v}{x_+}_X + \scl{v}{x_-}_X \leq \scl{v}{x_+}_X \leq \|v\|_X \|x_+\|_X \leq \|x_+\|_X,\]
where equality is attained with the choice $v = \tfrac{x_+}{\|x_+\|_X}$ since $\scl{x_+}{x_-}_X=0$.  Hence, $\sigma_{K_r}(x) = \|x_+\|_X$.

Let us show $\tfrac{x_+}{\|x_+\|_X}$ is the unique choice to reach equality. Indeed, if $v \in P_r$, $\|v\|_X \leq 1$ is such that $\scl{v}{x}_X = \|x_+\|_X$, then  
\[\|x_+\|_X = \scl{x_+}{v}_X \leq \|x_+\|_X \|v\|_X \leq \|x_+\|_X.\]
All these inequalities must be equalities, thus enforcing  $\|v\|_X =1$ on the one hand, and nonnegative linear dependence of $v$ and $x_+$ on the other hand, which amounts to $v = \tfrac{x_+}{\|x_+\|_X}$. We have proved $\partial (\sigma_{K_r})(x) = \{\tfrac{x_+}{\|x_+\|_X}\}$, and hence $ \sigma_{K_r}(x)\, \partial (\sigma_{K_r})(x) = \{x_+\}$.
\end{proof}

\cz{
\begin{lmm}
\label{unique-primal-closed-convex}
Let $K_r :=  P_r \cap \overline{B(0,1)}$. If $b \in \overline{A(P_r)}$, then for every $\e>0$, $(\mathcal{P}_{\e}^\mathrm{cl})$ admits a unique solution. 

If $b\in A(P_r)$, then $(\mathcal{P}_0^{\text{cl}})$ admits a unique solution.
\end{lmm}
\begin{proof}
In both cases, we know that the primal problems reach their infima, which we denote $\tfrac{1}{2}(\lambda^\star_\e)^2$ (resp.~$\frac{1}{2}(\lambda^\star)^2$). Since the objective functionals are convex, the sets of minimisers are convex, and are subsets of $\lambda^\star_\e S^1$ (resp. $\lambda^\star S^1$) where $S^1$ is the unit sphere of $X$. The latter set being strictly convex, both sets of minimisers are reduced to a singleton.
\end{proof}
}
Lemmas~\ref{intersect-ball} and~\ref{unique-primal-closed-convex} combined with Theorem~\ref{general-ncs} (resp. Theorem~\ref{convex-general-opt}) proves Theorem~\ref{intersect-ball-ncs} (resp. Theorem~\ref{intersect-ball-opt}).

\subsection{Optimality conditions and constructiveness}\label{subsec-constructivity-uniqueness}
Recall that, for a given $\e \geq 0$, a primal-dual optimal pair $(x^\star, y^\star)$ satisfies $x^\star \in \sigma_{K_r}(A^\ast y^\star)\, \partial  \sigma_{K_r}(A^\ast y^\star)$.
Most of what follows is dedicated to analysing this very inclusion in various cases. 
This will allow us to prove Proposition \ref{prop-opt-cond-ncs} and Theorem \ref{general-opt}.

\subsubsection{Cone generated by a bounded closed convex set}
\label{subsubsec-convex-constructive}
Here, we assume that $P_r = \cone(K_r)$ with $K_r$ a closed convex bounded set containing $0$. As a result, $P_r$ is convex but not necessarily closed. 

In this general setting, the set $\sigma_{K_r}(x)\, \partial (\sigma_{K_r})(x)$ is not necessarily reduced to a singleton for a given $x \in X$. 

\begin{lmm}
\label{extremal-singleton}
    Assume that~\eqref{extremal} holds. Then,
 \[\forall y \in Y, \quad \sigma_{K_r}(A^\ast y)\, \partial \sigma_{K_r}(A^\ast y) \text{ is reduced to a singleton}.\]    
\end{lmm}
\begin{proof}
Let $x \in X$.
Let us first remark that the latter set is never empty, because \[\partial \sigma_{K_r}(x) = \left\{v \in K_r, \; \langle x, v\rangle_X = \sup_{v\in K_r}\langle x, v\rangle_X\right\},\]
and the weakly continuous function $v \mapsto \langle x, v\rangle_X$ is bounded and attains its supremum on the weakly compact set $K_r$.

Since $\sigma_{K_r}(x) = 0$ if and only if $x \in P_r^\circ$, and by definition of the cone of singular vectors, we end up with the equivalence
\[\sigma_{K_r}(x)\, \partial \sigma_{K_r}(x) \text{ is reduced to a singleton} \quad \iff \quad x \in P_r^\circ \cup \sing(K_r)^c.\]

Under~\eqref{extremal}, we obtain the result by applying the above to $x = A^\ast y$ for $y \in Y$.
\end{proof}

This Lemma then immediately yields Proposition \ref{prop-opt-cond-ncs}.

\subsubsection{General cone}
\label{subsubsec-general-constructive}
Here, we assume that $P$ is any cone containing $0$, $K$ is a bounded set containing $0$ such that $\cone(K) = P$. Then, we let $K_r := \adhco(K)$ and $P_r = \cone(K_r)$. 
Recall that the assumption~\eqref{extremal} refers to the inclusion $\mathrm{Ran}(A^\ast) \subset P_r^\circ \cup \sing(K_r)^c$, while~\eqref{extremality} refers to the assumption $\ext(K_r)\subset K$.

\begin{lmm}
\label{extremal-extremality}
    Assume that~\eqref{extremal}-\eqref{extremality} holds. Then,
 \[\forall y \in Y, \quad \sigma_{K_r}(A^\ast y)\, \partial \sigma_{K_r}(A^\ast y) \text{ is reduced to a single element in $P$}.\]    
\end{lmm}
\begin{proof}
By Lemma~\ref{extremal-singleton}, we know that $\sigma_{K_r}(A^\ast y)\, \partial \sigma_{K_r}(A^\ast y)$ is reduced to a singleton. If $\sigma_{K_r}(A^\ast y) = 0$, this singleton equals $0$, which is an element of $P$. If $\sigma_{K_r}(A^\ast y) \neq 0$, then this singleton is a multiple of a unique vector $v$ in $K_r$ maximising the function $v \mapsto \langle A^\ast y, v\rangle_X$ over $K_r$. The set of maximisers of a convex function over a convex set must contain at least an extreme point. Hence, $v$ is in fact in $\ext(K_r)$, and by~\eqref{extremality}, it is an element of $K$. Consequently, $\sigma_{K_r}(A^\ast y)\, \partial (\sigma_{K_r})(A^\ast y)$ contains the single element $\sigma_{K_r}(A^\ast y)\, v$, which lies in the cone $\cone(K) = P$. 
\end{proof}

This lemma, combined with Corollary \ref{cor-single-x-in-subdiff} and Theorems \ref{intersect-ball-opt} and \ref{convex-general-opt} respectively, completes the proofs of Theorem~\ref{general-opt} $(i)$ and $(ii)$ respectively.

\cz{\section{Dual attainment and shared normal vectors}
\label{sec-constructiveness}}

In this Section, we discuss further the constructive character of our approach. Constructiveness is to be understood here in the sense that an exact (or approximate) solution $x^\star$ to \eqref{farkas-problem} can be characterised by 
\[x^\star \in \mathcal{F}(y^\star),\]
where $y^\star\in Y$, $\mathcal{F}$ is a set-valued mapping on $Y$, and both $y^\star$ and $\mathcal{F}$ can be approximated numerically. 
The situation is even more favorable if the identity above completely characterises $x^\star$ - \ie $\mathcal{F}(y^\star)$ is reduced to a singleton.

Given Theorems \ref{intersect-ball-opt}, \ref{convex-general-opt} and \ref{general-opt}, we will consider the more specific terminology of a \textbf{weakly constructive solution} for situations where \textit{the dual problem admits a minimiser}, which leads to a constructive characterisation by Theorems \ref{intersect-ball-opt}, \ref{convex-general-opt} and \ref{general-opt}. The more amenable situation where the saddle-point relation uniquely defines $x^\star$ corresponds to what we call a \textbf{constructive solution.} 

Let us first remark that, by Theorem \ref{general-opt}, our method is always weakly constructive for approximate resolution. 
We thus focus this section's study to the problem of exact resolution.

\

\subsection{Conditions for the dual problem to attain its infimum}

Recall that Theorems \ref{intersect-ball-ncs} and \ref{general-ncs} rely on an unusual application of Fenchel-Rockafellar's duality theorem, in which \textbf{regularity conditions on the dual problem} are assumed, leading to the solution of the primal problem being reached at a minimiser $x^\star$ if the infimum value of the primal problem is finite. This approach does not, however, tell us whether the dual problem reaches its infimum, as remarked in Theorem \ref{intersect-ball-opt}.

Still, it is clear from Theorems \ref{intersect-ball-opt} and \ref{general-opt} that if the dual problem admits a minimiser $y^\star$, then our method is weakly constructive through the saddle-point condition
\[x^\star \in \sigma_{K}(A^\ast y^\star) \, \partial \sigma_{K}(A^\ast y^\star).\]

It is therefore natural to explore sufficient (potentially necessary) conditions for which the dual functional reaches its infimum.
There is extensive literature on this matter. \cz{Let us mention, in close relation to our work, the articles \cite{zualinescu2008zero, zualinescu2012duality}, which study the constrained linear equation \eqref{farkas-problem} by considering a linear auxiliary optimisation problem and its dual. In these works, the author derives sufficient conditions for the dual to reach its infimum. In particular, one of these sufficient conditions is for \eqref{farkas-problem} to have a solution, along with a generalised Slater condition. }In a more general framework,~\cite{boct2008revisiting} provides a necessary and sufficient condition, but it requires knowing the infimum of the primal problem and hence usually is intractable from the practical point of view.

Various sufficient conditions have been proposed, relying on appropriate strong duality: under a \textit{regularity condition on the primal problem}, the primal and dual problems have the same value and the dual problem reaches its infimum if it is finite. The idea is thus to formulate weaker regularity conditions on the primal problem (using weakened notions of the relative interior) that still ensure strong duality see~\cite{Teboulle1990, boct2011comparison, boct2012regularity, BorweinRelativeInterior2003, borwein2010convex}.  Relying on the notions of quasi-relative interior (and strong quasi-relative interior, respectively),\footnote{For a subset $C$ of a Hilbert space $H$, the quasi-relative interior $\mathrm{qri}(C)$ is the set of vectors $x \in H$ such that $\overline{\cone(C - x)}$ is a subspace of $X$, while  the strong  quasi-relative interior $\mathrm{sqri}(C)$ is the set of vectors $x \in H$ such that $\cone(C - x)$ is a closed subspace of $X$.} one such rather weak condition is
$0 \in \mathrm{sqri}(P - \{x \in X, A x= b\})$~\cite{boct2012regularity}.
Some relevant settings covered by all these works are:
\begin{itemize}
    \item the case where $P = X$,
    \item when $X$ is of finite dimension, the case where $P$ is a subspace of $X$,
    \item when $Y$ is of finite dimension, the case where $b \in A\, \mathrm{qri}(P)$ provided that $\mathrm{qri}(P) \neq \emptyset$.
\end{itemize}

The first two conditions are somewhat restrictive, while the last condition is difficult to check in practice, since it requires knowledge about the set of solutions to the affine equation $Ax=b$ in relation to the cone~$P_r$, which is precisely what we hope to study.

Another possible way to proceed would be to exploit the saddle-point conditions directly:
\begin{lmm}\label{lem-saddle-exact-2}
   Let $K$ be a bounded closed convex set containing $0$, and $P=\cone(K)$. 
    Suppose that $b\in A(P)$, so that the primal problem admits a minimiser $x^\star$. Then, $y^\star\in Y$ is a minimiser to the dual problem if and only if
    \[x^\star\in \sigma_{K}(A^\ast y^\star) \, \partial \, \sigma_{K}(A^\ast y^\star).\]
\end{lmm}
\begin{proof}
    Given that $x^\star$ is a minimiser for the primal problem, $y^\star$ is a minimiser if and only if $(x^\star, y^\star)$ fulfill the saddle-point conditions \eqref{saddle-point}. Given that, by hypothesis, $Ax^\star=b$, we have $\partial \delta_{\{b\}}(Ax^\star)=\partial \delta_{\{b\}}(b)=Y. $ Hence, the second saddle-point condition is always satisfied, and the first saddle-point condition is a necessary and sufficient condition for $y^\star$ to be a minimiser for the dual problem.
\end{proof}
A sufficient condition for constructiveness immediately follows:
\begin{lmm}
Let $P$ be a closed convex cone, $K := P \cap \overline{B(0,1)}$. 

If $P \subset \ran(A^\ast)$, then for every $b\in A(P)$, the solution to \eqref{exact-Farkas-problem} is constructive. 
\end{lmm}
\begin{proof}
Let $b\in A(P)$, and denote $x^\star$ a solution to \eqref{exact-Farkas-problem}. By hypothesis, there exists $y^\star\in Y$ such that $x^\star=A^\ast y^\star$. Then,
   $x^\star = A^\ast  y^\star = (A^\ast  y^\star)_+ = \sigma_{K}(A^\ast y^\star)\, \partial  \sigma_{K}(A^\ast y^\star)$ by Lemma~\ref{intersect-ball}.
\end{proof}

For the remainder of this section, we shall endeavour to study the problem of minimising the dual functional, by finding vectors that satisfy the saddle-point condition. 
In the next subsection, we shall see how this can be understood through a geometric interpretation.

\subsection{Geometric interpretation: shared normal vectors}
In this section, unless otherwise stated, we operate under the following assumptions:
\begin{assum}\label{assum-constructive}
\
    \begin{itemize}
        \item $K\subset X$ is a bounded, closed and convex set containing $0$. We define, as usual, $P=\cone (K).$ Recall that $P$ is not necessarily closed. 
        \item $b\neq 0$.
        \item the primal problem $(\mathcal{P}_0)$ has a solution $x^\star\neq 0$, \ie $Ax^\star=b$ and
\[\pi_0=\inf_{Ax=b} \tfrac12j_K^2(x)=\tfrac12 j_K^2(x^\star)=:\tfrac12 \lambda^\star.\]
    \end{itemize}
    \end{assum}
Under the two first assumptions, and if $b \in A(P)$, then the last assumption is automatically satisfied as per Corollary~\ref{cor-equivalence-optim-exact}.

Note that we thus have $A^{-1}(b)=x^\star + \ker(A)$. Then, clearly, for any $x\in A^{-1}(b)$, we have $N_{A^{-1}(b)}(x)= (\mathrm{Ker}(A))^{\perp}=\overline{\ran (A^\ast)}$.

On the other hand, recall that by property of the support function, for any convex set $C\subset X$, and for any $v\in X$, 
\[\partial \sigma_C(v)=\argmax_{x\in C}\langle x, v \rangle_X.\]
Moreover, by definition of the normal cone, we have
\begin{equation}\label{subdiff-normal-cone}v\in N_C(x) \iff x\in \partial \sigma_C(v).\end{equation}

We now establish a proposition that gives us a necessary, and sometimes sufficient, condition for the dual problem to reach its infimum. It relies on the notion of shared normal vectors.
\cz{\begin{dfntn}[Shared normal vector]
Let $C_1,C_2\subset X$ be closed convex sets and $x\in C_1\cap C_2$.
Given a vector $v \in X$, we say that $C_1$ and $C_2$ \emph{share the normal vector} $v$ at $x$ if $v\in N_{C_1}(x)\cap N_{C_2}(x)$, i.e., the two sets admit a common supporting hyperplane
at the contact point $x$ with normal $v$.
\end{dfntn}
}
\begin{prpstn}\label{prop-dual-min-shared-normal}
Under Assumption \ref{assum-constructive}, the properties
\begin{itemize}
    \item[(i)] the dual problem reaches its infimum at $y^\star \in Y$,
    \item[(ii)] the convex sets $A^{-1}(b)$ and $\lambda^\star K$ share a nonzero normal vector $v\in X$ at their contact point $x^\star,$ such that the \cz{regularity condition} holds:
    \begin{equation}\label{non-degenerate-normal-vector}\langle v, x^\star \rangle= \sigma_{\lambda^\star K}(v) >0,\end{equation}
    \end{itemize}
    satisfy $(i) \implies(ii)$. Moreover, if $\ran (A^\ast)$ is closed, $(i) \iff (ii)$.
\end{prpstn}

\begin{proof}

We first prove $(i) \implies (ii)$. Assuming $(i)$, recall the optimality condition \eqref{saddle-point}
\[x^\star \in \sigma_K(A^\ast y^\star) \,\partial \sigma_K(A^\ast y^\star).\]
Strong duality (Proposition \ref{strong-duality}) implies
\[\tfrac12 (\lambda^\star)^2=\pi_0=-\min_Y J_0=-\tfrac12\sigma_K(A^\ast y^\star)^2+\langle b, y^\star\rangle=-\tfrac12\sigma_K(A^\ast y^\star)^2+\langle Ax^\star, y^\star\rangle.\]
The optimality condition above clearly implies that $\langle x^\star, A^\ast y^\star\rangle=\sigma_K(A^\ast y^\star)^2$ so that we have
\[\tfrac12(\lambda^\star)^2=\tfrac 12 \sigma_K(A^\ast y^\star)^2,\]
\ie, since $\sigma_K(A^\ast y^\star)\geq 0$ by property of $K$,
\[0<\lambda^\star=\sigma_K(A^\ast y^\star).\]
In particular, this yields the regularity condition \eqref{non-degenerate-normal-vector}.

We can then rewrite the optimality condition as
\[x^\star \in \partial \sigma_{\lambda^\star K}(A^\ast y^\star),\]
which, thanks to \eqref{subdiff-normal-cone}, implies that $A^\ast y^\star \in N_{\lambda^\star K}(x^\star)$. On the other hand, since $A^\ast y^\star \in\ran (A^\ast)\subset (\ker(A))^\perp=N_{A^{-1}b}(x^\star)$, this proves $(ii)$.

To prove the converse under the hypothesis that $\ran (A^\ast)$ is closed, first note that
$(\ker(A))^\perp =\overline{\ran(A^\ast)}=\ran (A^\ast)$.
Thus, if $\lambda^\star K$ and $A^{-1} (b)$ share a normal vector $v$, then there exists $y\in Y$ such that $v=A^\ast y$, and we can write, by definition of $N_{\lambda^\star K}(x^\star)$,
\[x^\star \in \partial \sigma_{\lambda^\star K}(A^\ast y)= \lambda^\star  \partial \sigma_{K}(A^\ast y).\]

Moreover, by hypothesis, $\sigma_{K}(A^\ast y)=\frac1{\lambda^\star}\langle A^\ast y, x^\star \rangle_X >0.$
We can then define 
\[y^\star:=\sqrt{\frac{\lambda^\star}{\sigma_{K}(A^\ast y)}} y.\]
Then,
\[x^\star\in \lambda^\star \partial \sigma_K (A^\ast y)=\lambda^\star \sqrt{\frac{\sigma_K(A^\ast y)}{\lambda^\star}} \, \partial\sigma_K(A^\ast y^\star)=\sigma_K(A^\ast y^\star)\, \partial\sigma_K(A^\ast y^\star).\]
Using the notations introduced when discussing Fenchel-Rockafellar duality, namely $F = \tfrac12 j_{K}^2$ and $G = \delta_{\{b\}}$, the above rewrites
\[x^\star \in \partial \big( \tfrac12 \sigma_{K}^2 \big)(A^\ast y^\star )=\partial F^\ast(A^\ast y^\star).\]
Hence 
\[A x^\star = b\in A \partial F^\ast(A^\ast y^\star), \]
which is equivalent to
\[0\in A\partial F^\ast(A^\ast y^\star)- \partial G^\ast(y^\star)=\partial \left( F^\ast \circ A^\ast + G^\ast(- \cdot)\right)(y^\star)=\partial J_0(y^\star).\]
Note that the above computation on subdifferentials is guaranteed by the sufficient stability condition used to establish strong duality~\cite{Rockafellar1967}, namely
\[ A^\ast\operatorname{dom}G^\ast(-\cdot) \cap \operatorname{cont}(F^\ast) \neq \emptyset. \]
Hence, the dual problem reaches its infimum at $y^\star$.
\end{proof}

\subsection{On the existence of shared normal vectors}

Given the above section, we now focus on the existence of shared normal vectors between the sets $\lambda^\star K$ and $A^{-1}b$. We begin with a counterexample in infinite dimension.
\paragraph{A counterexample in infinite dimension: convex cone (not closed), no shared normal vectors.}
We here here heavily draw from \cite[Example 6.11 iv]{bauschke-combettes}.

Let $X=\ell^2(\N)$, let $\{e_n\}_{n\in \N}$ be a Hilbert basis of $X$. We define the following convex, closed, bounded set:
\[C:=\left\{x=\sum_{n\in \N} x_n e_n, \quad -2^{-n} \leq x_n \leq 4^{-n}, \ \forall n \in \N.\right\}\]
This set generates a dense cone (see \cite{bauschke-combettes}), but not the whole space.
Define
\[x^\star:=\sum_{n \in \N} 2^{-n} e_n.\]
Then, $-x^\star\in C $ but $x^\star\notin C.$
Define the following generating set
\[K:=C+x^\star=\left\{x=\sum_{n\in \N}x_n e_n, \quad 0\leq x_n \leq 4^{-n}+2^{-n}, \ \forall n\in \N\right\},\]
which is clearly convex, bounded, closed, and contains $0$.
Moreover, $x^\star\in K$, $j_{K}(x^\star)=1$.

Let $\mathcal{I}\subset \N$ such that $\operatorname{card}(\mathcal{I})<+\infty$. Then, for any $A\in L(X,Y)$ such that $\ker(A)= \operatorname{Span}\{e_n, \ n\in \mathcal{I}\}$, 
\[\inf_{x\in x^\star+\ker (A)} \tfrac12 j_{K}^2(x)=\tfrac12 j_K^2(x^\star)=\tfrac12.\]
Hence, $K$, $A$, and $x^\star$ satisfy Assumption \ref{assum-constructive} with $\lambda^\star=1$.

Now, by density of $\cone (C)$,
\[(K-x^\star)^\circ=C^\circ=\{0\},\]
\ie
\[N_{K}(x^\star)=\{0\}.\]
Hence, 
\[N_{K}(x^\star) \cap N_{A^{-1}(b)}(x^\star)=\{0\},\]
\ie there are no nontrivial shared normal vectors, even though $x^\star$ minimises $\frac12j_{K}^2$ on $A^{-1}(b)$. By Proposition \ref{prop-dual-min-shared-normal}, the associated dual problem does not have a minimiser even though its infimum is finite.

\paragraph{An important partial result on the existence of shared normal vectors.}

The above counterexample precludes any general result on shared normal vectors in infinite dimension. However, in finite dimension, we are able to prove a positive result:
\begin{lmm}
Suppose that $X$ has finite dimension. 
     Then, under Assumption \ref{assum-constructive}, $\lambda ^\star K$ and $A^{-1}b$ share a nonzero normal vector in $x^\star$.
\end{lmm}
\begin{proof}
    We proceed by contradiction. Suppose that $\lambda^\star K$ and $A^{-1}b$ do not share a nonzero normal vector in $x^\star$: 
    \[N_{A^{-1}(b)}(x^\star) \cap N_{\lambda^\star K}(x^\star)=\{0\},\]
   \ie
    \[(\lambda^\star K - x^\star)^\circ \cap \operatorname{ran}(A^\ast)=\{0\}.\]
    Taking the cones generated by the sets involved above, we arrive at 
    \[\cone\left((\lambda^\star K - x^\star)^\circ\right) \cap \ran(A^\ast) = \{0\}.\]
    By the elementary result~\eqref{identities-cones}, we then have
    \[\left(\cone(\lambda^\star K -x^\star) \, + \, \ker(A)\right)^\circ= \{0\},\]
    \ie
    \[\overline{\cone(\lambda^\star K -x^\star) \, + \, \ker(A)}=X.\]

    In finite dimension, it is standard that a dense convex set is the whole space. Indeed, if it were not the whole space, it would be included in a closed half-space~\cite{Rockafellar1970}[Corollary 11.5.2], and such a half-space cannot be dense.

    Hence, we obtain that for any $\nu >0$, $\nu x^\star \in \cone(\lambda^\star K -x^\star) \, + \, \ker(A)$. That is, there exist $\rho \geq 0, \, k\in K, \, v\in \ker(A)$ such that
    \[\nu x^\star=\rho(\lambda^\star k - x^\star)+v,\]
    \ie
    \[(\nu+\rho) x^\star=\rho\lambda^\star k + v,\]
    \ie
    \[x^\star+v^\prime=\frac{\rho}{\nu+\rho}\lambda^\star k, \quad v^\prime \in \ker(A),\]
    which implies
    \[A(x^\star + v^\prime)=b, \quad j_{K}(x^\star +v^\prime)=\frac{\rho}{\rho+\nu}\lambda^\star < \lambda^\star=\inf_{Ax=b} j_{K}(x),\]
    we thus arrive at a contradiction.

\end{proof}

\subsection{On the regularity condition}

Notice that, in Proposition \ref{prop-dual-min-shared-normal}, the condition $\langle v, x^\star \rangle>0$ for the shared normal vector $v$ might not seem very natural in the context of a geometric interpretation. We now show a simple counterexample, illustrating how the unfavorable situation $\langle v, x^\star \rangle_X =0$ can in fact occur even in simple settings.

\paragraph{Shared normal vectors, no minimiser for the dual functional (closed cone)} 
Take $X=\R^3$, $Y = \R^2$, $\alpha>0$, and let $P$ be the convex closed cone given by
\[P:=\left\{(x_1, x_2, x_3)\in \R^3, \quad x_3\geq0, \ x_1^2+x_2^2 \leq \alpha^2 x_3^2\right\},\]
and denote as usual $K:=P\cap \overline{B(0,1)}$.
One easily checks that
\[P^\circ=\left\{(x_1, x_2, x_3)\in \R^3, \quad x_3\leq0, \ x_1^2+x_2^2 \leq \tfrac1{\alpha^2} x_3^2\right\}.\]
Now define
\[A:=\begin{pmatrix}
    1 & 0 & -\alpha \\
    \alpha & 1+\alpha^2 & 1 
\end{pmatrix}, \quad b=\begin{pmatrix}
    0 \\1
\end{pmatrix}.
\]
One easily checks that
\[x^\star:=(\tfrac{\alpha}{1+\alpha^2}, 0, \tfrac1{1+\alpha^2})\, \in P, \quad Ax^\star=b.\]
Moreover, given
the set of solutions
\[A^{-1}(b)=x^\star + \ker(A)=\left\{\big(\tfrac{\alpha}{1+\alpha^2}-\alpha s, s, \tfrac1{1+\alpha^2}-s \big), \ s\in \R\right\},\]
we have
\[A^{-1}(b)\cap P = \{x^\star\}.\]
Hence, $x^\star$ is also a (in fact, the) solution to the primal problem \eqref{primal-fct}.

Let us here denote $\lambda^\star=\|x^\star\|_X.$
We have on the one hand
\[\ran (A^\ast) = \operatorname{Span}\{(1,0,-\alpha) ; (\alpha, 1+\alpha^2, 1)\},\]
and on the other hand, using Lemma \ref{lem-charac-normal-cone},
\[N_{\lambda^\star K}(x^\star)=\cone(\{x^\star\})+ P^\circ \cap \{x^\star\}^\perp.\]
Now, it is clear that
\[P^\circ \cap \{x^\star\}^\perp=\cone(\{(1,0,-\alpha)\})\subset \ran (A^\ast).\]
Moreover, $x^\star \notin \ran (A^\ast)$, so that finally
\[N_{\lambda^\star K}(x^\star) \cap \ran (A^\ast)=P^\circ \cap \{x^\star\}^\perp.\]
In particular, all shared normal vectors of $A^{-1}(b)$ and $\lambda^\star K$ are normal to $x^\star$, which by Proposition \ref{prop-dual-min-shared-normal} implies that the dual problem associated to $b$ does not have a minimiser, even though its infimum is finite.



\paragraph{A sufficient condition.} In the above counterexample, $\lambda^\star K$ and $\ran (A^\ast)$ share nontrivial normal vectors, but they are all orthogonal to $x^\star$. We now give a sufficient condition under which this type of situation cannot arise. Its proof is based on a separation result between two well-chosen convex sets. 
\begin{prpstn}
    We work under Assumption \ref{assum-constructive}, and assume furthermore that $P$ is closed, that $K=P\cap \overline{B(0,1)}$, and that $\ran(A^\ast)$ is closed.

    If for all $x\in P$, $\operatorname{Span}(\{x\})+P$ is closed, then the dual problem always admits a minimiser.
\end{prpstn}

\begin{proof}
Since $\operatorname{ran}(A^\ast)$ is closed, by Proposition \ref{prop-dual-min-shared-normal}, the dual problem admits a minimiser if and only if there exists a vector $v\in N_{\lambda^\star K}(x^\star)\cap\ran(A^\ast)$ such that $\langle v, x^\star\rangle_X >0.$
Proceed by contradiction: suppose 
\[\ran (A^\ast) \cap \{v\in N_{\lambda^\star K}(x^\star), \ \langle v, x^\star\rangle>0\}=\emptyset.\]
Since, for any $v\in N_{\lambda^\star K}(x^\star)$, $\langle v, x^\star \rangle_X \geq 0$, this also writes
\begin{equation}\label{suppose-all-normal}N_{\lambda^\star K}(x^\star)\cap\ran(A^\ast)=N_{\lambda^\star K}(x^\star)\cap N_{A^{-1}(b)}(x^\star) \subset \{x^\star\}.^\perp\end{equation}

Recall now that, by Lemma \ref{lem-charac-normal-cone}, since $\|x^\star\|_X=\lambda^\star$,
\[N_{\lambda^\star K}(x^\star)=\cone(\{x^\star\})+P^\circ\cap \{x^\star\}^\perp\]
so
\[ \{v\in N_{\lambda^\star K}(x^\star), \ \langle v, x\rangle>0\}=\{\rho x^\star +v, \ \rho>0, \ v\in P^\circ \cap \{x^\star\}^\perp\},\]
and, by linearity, our assumption implies
\[\ran (A^\ast) \cap (x^\star+\overline{B(0,1)} \cap P^\circ\cap \{x^\star\}^\perp)=\emptyset.\]
Now since the sets in the above intersection are convex and closed, and $x^\star+\overline{B(0,1)} \cap P^\circ\cap \{x^\star\}^\perp$ is (weakly) compact, we can use a strict separation theorem: there exists $h\in X$, $h \neq 0$ and $\alpha>0$ such that
\begin{equation}\label{kerA-separates-normalplus}\forall v\in \overline{B(0,1)}\cap P^\circ\cap \{x^\star\}^\perp, \; \forall y\in Y, \quad \langle h, x^\star+v \rangle_X \leq \langle h, A^\ast y \rangle_X-\alpha.\end{equation}
For the above inequality to hold, one must have $h\in \ran (A^\ast)^\perp=\ker (A)$, which leads to
\[\forall v\in \overline{B(0,1)}\cap P^\circ\cap \{x^\star\}^\perp, \quad \langle h, x^\star + v \rangle_X \leq -\alpha. \]
Since $\cone(x^\star+\overline{B(0,1)}\cap P^\circ \cap \{x^\star\}^\perp)\subset N_{\lambda^\star K}(x^\star)$, this implies that $h\in N_{\lambda^\star K}(x^\star)^\circ$.

Now, using in order Lemma~\ref{lem-charac-normal-cone} and the identities~\eqref{identities-cones} (since $P$ and $\{x^\star\}^\perp$ are closed cones),
\[N_{\lambda^\star K}(x^\star)^\circ=(\cone(\{x^\star\})+P^\circ\cap \{x^\star\}^\perp)^\circ=\cone(\{x^\star\})^\circ \cap (P^\circ\cap \{x^\star\}^\perp)^\circ=\cone(\{x^\star\})^\circ \cap  \overline{\operatorname{Span}(\{x^\star\})+P}.\]
By hypothesis, $\operatorname{Span}(\{x^\star\})+P$ is closed, so we obtain
\[h\in \ker(A) \cap \cone(\{x^\star\})^\circ \cap  (\operatorname{Span}(\{x^\star\})+P).\]
Now, without loss of generality, we write $h=\rho x^\star + x_0$, with $\rho\in  [-1,1]$ and $x_0\in P.$
Then, for $s\in[0,1]$, $1+s\rho\geq 0$ and
\[x^\star+sh=(1+s\rho)x^\star+sx_0 \in P\cap A^{-1}(b),\]
since $P+P\subset P$, by definition of a convex cone.
Moreover, considering
$f:s\mapsto \|x^\star+sh\|_X^2$, we find that
$f^\prime (0)=2\langle x^\star, h \rangle_X \leq -\alpha<0$, from \eqref{kerA-separates-normalplus}. Hence, for $s$ small enough, $\|x^\star+sh\|_X^2\leq \|x^\star\|_X$, and $x^\star+sh \in P\cap A^{-1}(b) $, which contradicts the definition of $x^\star.$

\end{proof}

\appendix 
\section{Elementary results in convex analysis}
We here gather several basic results used throughout the work. We do not claim any originality, but provide them here for completeness and readability. 

Throughout this subsection, we let $H$ be a Hilbert space and $C$ be a non-empty, bounded, closed and convex set containing $0$.
\label{app-gauge-results}
\begin{lmm}
\label{basic-gauge}
For any $u \in H$, one has
\begin{itemize}
    \item[(i)] $j_C(u) = 0 \iff u = 0$,
    \item[(ii)] for any $\alpha>j_C(u)$, $u \in \alpha C$,
    \item[(iii)] for any $\alpha>0$, $j_C(u) \leq \alpha \, \iff \, u \in \alpha C$,
    \item[(iv)] $u \in \cone(C) \, \iff \, u \in j_C(u)\,  C$.
\end{itemize}
\end{lmm}
\begin{proof}
    If $j_C(u) = 0$, then we may find $(\e_n)$ with $\e_n>0, \, \e_n \to 0$  and $(c_n) \in C^\N$ such that for all $n \in \N$, $u = \e_n c_n$. Since $C$ is bounded, we uncover $u=0$ at the limit $n \to+\infty$. The converse assumption is trivial since $0 \in C$, so (i) is proved.

    Since $\alpha> j_C(u)$, we may find $j_C(u) < \beta \leq \alpha$ such that $u \in \beta C$. Hence there exists $c \in C$ such that $u = \beta c = \alpha(\tfrac{\beta}{\alpha} c) = \alpha(\tfrac{\beta}{\alpha}c + (1-\tfrac{\beta}{\alpha}) 0)$. The vector $\tfrac{\beta}{\alpha}c$ appears as the convex combination of $c \in C$ and $0 \in C$, which shows that $u \in \alpha C$ and proves (ii).

    Let $\alpha>0$ and $u \in H$. If $u \in \alpha C$, then $j_C(u) \leq \alpha$ by definition. Conversely, assume $j_C(u) \leq \alpha$. We may find a sequence $(\alpha_n)$ with $\alpha_n > \alpha$ converging to $\alpha$. Since $\alpha_n > j_C(u)$, (ii) ensures the existence of $c_n \in C$ such that $u = \alpha_n c_n \iff c_n = \alpha_{n}^{-1} u$. Passing to the limit $n \to +\infty$ (since $\alpha>0$), we find that $(c_n)$ converges to $c:= \alpha^{-1} u$. By closedness of $C$, $c \in C$, showing that $u \in \alpha C$ and finishing the proof of~(iii).

    Let $u \in H$. If $u=0$, the equivalence is clear: $0 \in \cone(C)$ by the assumption $0 \in C$, and since $j_C(u)=0$, any $v \in C$ is such that $0 = u = j_C(u) v = 0$. Now assume that $u \neq 0$. By (i), we know that $j_C(u) \neq 0$, hence if $u = j_C(u) v$ with $v \in C$, we have $u \in \cone(C)$. Conversely, if $u \in \cone(C)$, we use (iii) with $\alpha = j_C(u)>0$, showing that $u \in \alpha C = j_C(u) C$, as wanted.%

\end{proof}

\begin{lmm}
\label{conj-supp-square}
There hold
\begin{itemize}
\item[(i)] $j_C \in \Gamma_0(H)$,
  \item[(ii)] $\tfrac{1}{2}\sigma^2_C \in \Gamma_0(H)$, $\tfrac{1}{2} j_C^2 \in \Gamma_0(H)$  and $(\tfrac{1}{2}j^2_C)^\ast = \tfrac{1}{2} \sigma_C^2$.
\end{itemize}
\end{lmm}
\begin{proof}
Since $j_C \geq 0$ and $j_C(0) = 0$, $j_C$ is proper. The convexity of $j_C$ is standard, see for instance~\cite{bauschke-combettes} (Example 8.27). Let us prove that $j_C$ is lower semicontinuous. Given $\alpha \geq 0$, we shall establish that $\{x \in H, \, j_C(x) \leq \alpha\}$, is closed. let $(x_n)$ be a sequence taking values in the previous set, converging to $x \in H$. By Lemma~\ref{basic-gauge} (iv), we may write $x_n = j_C(x_n) v_n$, with $(v_n) \in C^\N$. Upon extracting a subsequence, we may assume that $(v_n)$ converges weakly to some $v \in C$, by sequential weak compactness of $C$, and that the bounded sequence $(j_C(x_n))$ converges to some $\lambda \geq 0$ with $\lambda \leq\alpha$. Hence, $(x_n)$ converges weakly to $\lambda \, v$, which by uniqueness of the weak limit entails $x = \lambda v$, showing that $j_C(x) = \lambda \leq \alpha$, as wanted.

Let us now address (ii). We rather show that $(\tfrac{1}{2}\sigma^2_C)^\ast = \tfrac{1}{2} j_C^2$, which will prove the identity by Fenchel-Moreau's theorem. Since $0 \in C$, $\sigma_C \geq 0$. If $f \in \Gamma_0(H), \, f\geq 0$, then $f^2 \in \Gamma_0(H)$. Hence $\tfrac{1}{2}\sigma_C^2 \in \Gamma_0(H)$ and $\tfrac{1}{2} j_C^2 \in \Gamma_0(H)$.

Now let us prove the equality. For $x = 0$, it holds true since $j_C(0) = 0$ and $(\tfrac{1}{2}\sigma^2_C)^\ast(0) = -\tfrac{1}{2} \inf_{y \in H} \sigma^2_C(y) = 0$.

We now assume that $x \neq 0$. For $f \in \Gamma_0(H)$ and  $g \in \Gamma_0(\R)$ non-decreasing, we recall the composition formula
\begin{equation*}
(g\circ f)^\ast(x)= \inf_{\alpha \geq 0} \left(g^\ast(\alpha)+\alpha f^\ast\Big(\frac{x}{\alpha}\Big)\right),\end{equation*}
with the convention that for $\alpha = 0$, $0\, f^\ast(\tfrac{y}{0}) = \sigma_{\mathrm{dom}(f)}(y)$. Hence for $x \in H \setminus \{0\}$, one finds
\begin{align*}
\left(\tfrac{1}{2} \sigma_{C}^2\right)^\ast(x) & = \inf_{\alpha \geq 0}\left(\tfrac{1}{2} \alpha^2 + \alpha \delta_{C}\big(\frac{x}{\alpha}\big)\right) 
 = \inf_{\alpha > 0}\left(\tfrac{1}{2} \alpha^2 + \alpha \delta_{C}\big(\frac{x}{\alpha}\big)\right) \\
& = \tfrac{1}{2}\Big(\inf_{\alpha>0, x \in \alpha C}  \alpha\Big)^2 = \tfrac{1}{2} j_{C}^2(x),
\end{align*}
where we discarded $\alpha = 0$ since \[0 \,\delta_{C}\big(\frac{x}{0}\big)= \sigma_{\mathrm{dom}(\sigma_{C})}(x) = \sigma_H(x) =\delta_{\{0\}}(x) = +\infty,\]
using that $\mathrm{dom}(\sigma_{C}) = H$, by boundedness of $C$.
\end{proof}

\begin{lmm}
\label{closed_generated_cone}
If $0$ is in the interior of $C$ relative to the cone it generates, i.e., if
\[\exists \delta>0, \quad  \cone(C) \cap \overline{B(0,\delta)} \subset C,
\]
then $\cone(C)$ is closed.
\end{lmm}
\begin{proof}
Let $(p_n) \in \cone(C)^{\N}$ be a convergent subsequence, of limit $p \in H$.
By Lemma~\ref{basic-gauge}(iv), we may write $p_n = \lambda_n c_n$ with $(c_n) \in C^\N$ and $\lambda_n = j_C(p_n)$. If $p=0$, there is nothing to prove, so we may assume $p \neq 0$. Since $(c_n)$ is bounded by hypothesis, and $(p_n)$ converges to $p \neq 0$, $\liminf \lambda_n > 0$. Hence, if $(\lambda_n)$ is upper bounded, then upon extraction we may write $\lambda_n \to \lambda$ with $\lambda \neq 0$. In this case, we find that $(c_n)$ converges to $c :=\tfrac{p}{\lambda}$. Since $C$ is closed, we have $c \in C$ and hence $p = \lambda c \in \cone(C)$.

To conclude, we only need to show that $(\lambda_n)$ cannot have a diverging subsequence. By contradiction, assume that it is the case: upon extraction, we may assume that $\lambda_n \to +\infty$ as $n \to +\infty$.
Then we may form the sequence $w_n = \lambda_n^{-1/2}p_n = \lambda_n^{1/2} c_n$. This sequence satisfies $w_n \to 0$ as $n \to +\infty$ as well as $(w_n) \in \cone(C)^{\N}$. We also compute $j_C(w_n) = \lambda_n^{-1/2} j_C(p_n) = \lambda_n^{1/2} \to +\infty$ as $n \to +\infty$. Hence for $n$ large enough we have both $w_n\in \cone(C) \cap \overline{B(0,\delta)}$ and $w_n \notin C$, contradicting the assumption that $0$ is the interior of $C$ relative to $\cone(C)$.
\end{proof}

Finally, we prove a technical Lemma which will help us characterise various situations with a \textbf{closed} cone~$P$.
\begin{lmm}\label{lem-charac-normal-cone}
Let $P$ be a closed convex cone, and $K:=P\cap \overline{B(0,1)}$. For $\lambda>0$ and $x\in \lambda K$, we have the following characterisation for the set of normal vectors $N_{\lambda K}(x)$: 
    \[N_{\lambda K}(x)=\begin{cases}
        \cone(\{x\})+ P^\circ \cap \{x\}^\perp &\text{if} \ j_K(x)=\lambda, \\
        P^\circ \cap \{x\}^\perp &\text{otherwise.}
    \end{cases}\]
\end{lmm}
\begin{proof}
    Recall that, using the definition of $\sigma_K$ and Lemma \ref{intersect-ball},
    \[\begin{aligned}v\in N_{\lambda K}(x) &\iff \ x\in \partial \sigma_{\lambda K}(v) \\
    &\iff \ \lambda \sigma_K(v)=\langle x, v \rangle_X \\
    &\iff \ \lambda \|v_+\|_X=\langle v_+, x\rangle_X + \langle v_-, x \rangle_X.
    \end{aligned}\]
    Now, for $v\in X$, $\langle v_-, x \rangle_X \leq 0$ because $x \in P$, $v_- \in P^\circ$. Moreover, if $x\in \partial \sigma_{\lambda K}(v)$, then $\|x\|_X\leq \lambda$ since $K \subset \overline{B(0,1)}$, and 
    \[\lambda \|v_+\|_X=\sigma_{\lambda K}(v)=\langle v, x \rangle_X = \langle v_+, x\rangle_X + \langle v_-, x \rangle_X \leq \lambda \|v_+\|_X +\langle v_-, x\rangle_X  \leq \lambda \|v_+\|_X.\]
    Hence, the inequalities above are equalities, which yields
    \[v \in N_{\lambda K}(x) \iff \langle v_-, x \rangle_X =0, \quad \lambda \|v_+\|_X=\langle v_+, x \rangle_X. \]
    Now, $j_K(x)\leq \lambda$. If this inequality is strict, then $\lambda \|v_+\|_X=\langle v_+, x \rangle_X \iff v_+=0 $. Otherwise, $\|x\|_X=j_K(x)=\lambda$ and, since $v_+\in P$, $\lambda \|v_+\|_X=\langle v_+, x \rangle_X \iff \exists \rho\geq 0, \ v=\rho x$, which concludes the proof.
\end{proof}
\bibliographystyle{plain}
\bibliography{bib}
\end{document}